\documentclass[11pt,reqno]{amsart}
\usepackage{hyperref}
\hypersetup{
	colorlinks=true,
	linkcolor=blue,
	filecolor=blue,
	urlcolor=red,
	citecolor=green,}

\usepackage{CJK}
\newtheorem{thm}{Theorem}[section]
\newtheorem{cor}[thm]{Corollary}
\newtheorem{lem}[thm]{Lemma}
\newtheorem{prop}[thm]{Proposition}

\newtheorem{rem}{Remark}

\theoremstyle{definition}
\newtheorem{defn}[thm]{Definition}
\theoremstyle{remark}

\numberwithin{equation}{section}

\usepackage{bm}

\newcommand{\thmref}[1]{Theorem~\ref{#1}}

\newcommand{\lemref}[1]{Lemma~\ref{#1}}

\newcommand{\propref}[1]{Proposition~\ref{#1}}
\newcommand{\corref}[1]{Corollory~\ref{#1}}

\newcommand{\R}{\mathbb{R}}

\newcommand{\Lie}{\mathcal{L}}

\newcommand{\ric}{{\rm Ric}}

\newcommand{\di}{{\rm div}}

\newcommand{\la}{\bm{\langle}}
\newcommand{\ra}{\bm{\rangle}}

\newcommand{\vol}{{\rm{vol}}}

\begin{document}
	
	\title[Gradient estimates for $\Delta_pu-|\nabla u|^q+b(x)|u|^{r-1}u=0$]{Gradient estimates for $\Delta_pu-|\nabla u|^q+b(x)|u|^{r-1}u=0$ on a complete Riemannian manifold and Liouville type theorems}
	

	\author{Dong Han}
	\address{Institute of Mathematics, Academy of Mathematics and Systems Science, Chinese Academy of Sciences, Beijing
		100190, China}
	\email{handong181@mails.ucas.ac.cn}
	\author{Jie He}
	\address{College of Mathematics and Physics, Beijing University of Chemical Technology,  Chaoyang District, Beijing 100029, China}
	\email{hejie@amss.ac.cn}
	
	\author{Youde Wang*}
	\thanks{*Corresponding author}
	\address{1. School of Mathematics and Information Sciences, Guangzhou University; 2. Hua Loo-Keng Key Laboratory
		of Mathematics, Institute of Mathematics, Academy of Mathematics and Systems Science, Chinese Academy
		of Sciences, Beijing 100190, China; 3. School of Mathematical Sciences, University of Chinese Academy of Sciences,
		Beijing 100049, China.}
	\email{wyd@math.ac.cn}
	
	
	
	
	
	\begin{abstract}
		In this paper the Nash-Moser iteration method is used to study the gradient estimates of solutions to the quasilinear elliptic equation $\Delta_p u-|\nabla u|^q+b(x)|u|^{r-1}u=0$ defined on a complete Riemannian manifold $(M,g)$. When $b(x)\equiv0$, a unified Cheng-Yau type estimate of the solutions to this equation is derived. Regardless of whether this equation is defined on a manifold or a region of Euclidean space, certain technical and geometric conditions posed in \cite[Theorem E, F]{MR3261111} are weakened and hence some of the estimates due to Bidaut-V\'eron, Garcia-Huidobro and V\'eron (see \cite[Theorem E, F]{MR3261111}) are improved. In addition, we extend their results to the case $p>n=\dim(M)$. When $b(x)$ does not vanish, we can also extend some estimates for positive solutions to the above equation defined on a region of the Euclidean space due to Filippucci-Sun-Zheng \cite{filippucci2022priori} to arbitrary solutions to this equation on a complete Riemannian manifold. Even in the case of Euclidean space, the estimates for positive solutions in \cite{filippucci2022priori} and our results can not cover each other.
		
	\end{abstract}
	\maketitle
	\tableofcontents
	
	\section{Introduction}
	Gradient estimates are a fundamental and powerful tool in the analysis of partial differential equations on Riemannian manifolds. They can be used to deduce Liouville-type theorems (\cite{bidaut2020priori, bidaut2020DCDS, MR0431040, MR0834612, MR615628,WW}), obtain Harnack inequalities (\cite{MR0431040, MR0834612}), derive the estimate of the spectrum of manifolds(\cite{MR3275651, MR1333601}), and investigate the geometry of manifolds (\cite{MR2880214, MR3866881,MR2518892}). This topic has attracted the attention of many mathematicians.
	
	The present paper focuses on the following quasilinear equation
	\begin{align}\label{equ0}
		\Delta_pu-|\nabla u|^q+b(x) |u|^{r-1}u=0
	\end{align}
	defined on a complete Riemannian manifold $(M,g)$, which is the generalization of the equations \eqref{hamjoco}, \eqref{equa:2} and \eqref{psobo2}.
	
	If $b(x)$ is equal to zero, then the above equation \eqref{equ0} reduces to \begin{align}\label{hamjoco} \Delta_p u - |\nabla u| ^q = 0, \end{align} which is known as the quasilinear Hamilton-Jacobi equation. Using the well-known Bernstein technique, Lions \cite{MR833413} demonstrated that any $C^2$ solution to \eqref{hamjoco} on $\R^n$ with $q>1$ and $p=2$ must be constant. Recently, Bidaut-V\'eron, Garcia-Huidobro and V\'{e}ron \cite{MR3261111} discussed \eqref{equ0} in the case $n\geq p>1$ and established the following gradient estimate of the solutions to \eqref{hamjoco} regarding the distance to the boundary
	\begin{align}\label{jfathma}
		|\nabla u(x)|\leq c(n,p,q)(d(x,\partial\Omega))^{-\frac{1}{q+1-p}},\quad \forall x\in\Omega\subset\mathbb{R}^n,
	\end{align}
	and obtained some Liouville-type theorems. Moreover, they extended their estimates to solutions of equation \eqref{hamjoco} on complete non-compact manifolds $(M,g)$ with a lower bound depending on the Ricci curvature and sectional curvature etc. Concretely, for the case $q>p-1> 0$, if the Ricci curvature of $(M, g)$ satisfies $\mathrm{Ric}_g\geq-(n-1)B^2g$ and in addition
	$$
	\begin{cases}
		\mathrm{Sec}_g\geq -\widetilde{B}^2,&\text{ when } p>2;\\
		r_M(z)\geq d(z,\partial\Omega),& \text{ when } 1<p<2,
	\end{cases}
	$$
	where $\mathrm{Sec_g}$ is the sectional curvature and $r_M(z)$ is the convexity radius at the point $z\in M$, then any $C^1$ solution $u$ of \eqref{hamjoco} in $\Omega\subset M$ satisfies
	\begin{align}\label{jfathmd}
		|\nabla u(x)|^2\leq c(n,p,q)\max\left\{ B^{\frac{2}{q+1-p}},\, (1+ B_pd(x,\partial\Omega))^{\frac{1}{q+1-p}}
		d(x,\partial\Omega)^{-\frac{2}{q+1-p}}\right\},
	\end{align}
	where $B_p = B+(p-2)_+\widetilde{B}$.
	
	When $q=p$ in \eqref{hamjoco}, it is just the logarithmic transformation of the  $p$-Laplace equation.  In 2011,  Wang and Zhang \cite{MR2880214} used the Nash-Moser iteration method to study the $p$-harmonic functions on a complete Riemannian manifold with Ricci curvature bounded from below and obtained a local Cheng-Yau type gradient estimate.
	
	For any constant $k>0$, by a dilation tranformation $u=-kv $ of equation \eqref{equ0}, we can see that $v$ satisfies the following equation
	\begin{align*}
		\Delta_pv+k^{q-p+1}|\nabla v|^{q-p+1}-b(x)k^{r-p+1}|v|^{r-1}v=0.
	\end{align*}
	If we let $b(x)$ be a suitable constant function in the above equation and still denote $v$ by $u$, then we have
	\begin{align}
		\label{equa:2}
		\Delta_pu+N|\nabla u|^q+ |u|^{r-1}u=0,
	\end{align}
	where $N$ is a real constant.
	
	In the case $p=2$ in the above \eqref{equa:2},  Bidaut-V\'{e}ron, Garcia-Huidobro and V\'{e}ron (see \cite{bidaut2020priori, bidaut2020DCDS}) also studied the following semilinear equation
	\begin{align}
		\label{equation:1.6}
		\begin{cases}
			\Delta u + N|\nabla u|^q + |u|^{r-1}u=0\quad\text{in}\,\, \Omega\subset\R^n;\\
			n\geq 1, ~ r>1, ~ q >\frac{2r}{r+1}, ~ N > 0.
		\end{cases}
	\end{align}
	By using a delicate combination of refined Bernstein techniques and Keller-Osserman estimate, they proved that any solution $u$ of equation \eqref{equation:1.6} in a domain $\Omega\subset\R^n$ satisfies
	\begin{align*}
		|\nabla u(x)|\leq c(n,q,r)\left( N^{-\frac{r+1}{(r+1)q-2r}}+ (Nd(x,\partial\Omega))^{-\frac{1}{q-1}}\right), \quad\forall x\in\Omega.
	\end{align*}
	They also obtained a gradient estimate and Liouville theorems for positive solutions of \eqref{equa:2} when $n\geq 2, 1<r<(n+3)/(n-1), q<(n+2)/n$ and $N>0$.
	
	Later, Filippucci, Sun and Zheng \cite{filippucci2022priori} generalized in 2022 the results in \cite{bidaut2020priori} to the $p$-Laplace case of equation \eqref{equa:2}. In the case $N>0$, $r>\max\{p-1, 1\}$ and $q>pr/(r+1)$, they showed that there exists a positive constant $c(n,p,q,r)$ such that any positive solution to \eqref{equa:2} on $\Omega\subset \R^n$ satisfies
	\begin{align*}
		|\nabla u(x)|\leq c(n,p,q,r)\left( N^{-\frac{r+1}{(r+1)q-pr}}+(Nd(x,\partial\Omega))^{-\frac{1}{q-p+1}}\right), \quad\forall x\in\Omega.
	\end{align*}
	The arguments in \cite{ filippucci2022priori, bidaut2020priori} depend strongly on the invariance of translation of \eqref{equa:2} on $\R^n$ and the geometric symmetric structure of Euclidean space.
	
	Another related equation is the generalized Lane-Emden equation
	\begin{align}\label{lanemden}
		\begin{cases}
			\Delta_p u + u^r=0 \quad \text{in}\quad \mathbb R^n;\\
			u>0;\\
			n>p>1,\quad r>0,
		\end{cases}
	\end{align}
	which plays an important role on modeling meteorological or astrophysical phenomena \cite{CCK, CG, CK}. Equation \eqref{lanemden} has been widely studied in the literature \cite{ MR615628, MR982351, MR1134481, MR1004713, MR1121147}. Serrin and Zou (\cite{MR1946918}) proved that equation \eqref{lanemden} has a solution if and only if $r<np/(n-p)-1$.
	
	Very recently, inspired by the work of Wang and Zhang \cite{MR2880214}, Wang and Wei \cite{WW} adopted the Moser iteration to study the nonexistence of positive solutions to the semilinear elliptic equation
	$$\Delta u+au^{r}=0$$
	defined on a complete Riemannian manifold $(M, g)$ with $\dim(M)=n$, where $a$ is a positive constant. It is shown in \cite{WW} that, if the Ricci curvature of the manifold is nonnegative and $$r\in\left(-\infty,\quad  \frac{n+1}{n-1}+\frac{2}{\sqrt{n(n-1)}}\right),$$
	then the above equation does not admit any positive solution. Later, He, Wang and Wei(\cite{He-Wang-Wei}) generalized this result to the equation $\Delta_pu+au^r=0$ and improved the range of $r$ in which the gradient estimates and Liouville type theorems hold true. The results in \cite{WW, He-Wang-Wei} are an extension on Riemannian manifolds of the Liouville property of the Lane-Emden equation on an Euclidean space and some restriction assumptions in the previous work, for instance $r>0$ and $p\leq n$, were removed.
	
	On the other hand, the equation \eqref{equ0} is also related to the $L^p$-log Sobolev inequality. Chung and Yau \cite{chung1996logarithmic} showed that the critical function of the $L^2$-log-Sobolev inequality
	$$C_{M}\int_M u^2\log u^2 dV\leq \int_M |\nabla u|^2dV$$
	satisfies the following elliptic equation
	$$\Delta u + C_M u \log u^2 = 0.$$
	Motivated by Chung-Yau \cite{chung1996logarithmic} (see also \cite{DD, G}), one also considered the critical functions which achieve the Sobolev constant $C_{LS}$ of the $L^p$-log-Sobolev inequality
	$$
	C_{LS}\int_M u^p\log u^p dV\leq \int_M |\nabla u|^pdV,\quad  p>1, \quad u>0.
	$$
	It is not difficult to see that the critical functions satisfy the following Euler-Lagrange equation,
	\begin{align*}
		\Delta_pu + C_{LS}u^{p-1}\log u^p=0.
	\end{align*}
	By a logarithmic transformation $v=-(p-1)\log u $, the above equation becomes
	\begin{align}\label{psobo2}
		\Delta_pv-|\nabla v|^p+bv=0,
	\end{align}
	where $b = p(p-1)^{p-2}C_{LS}$. In particular, Wang and Xue \cite{MR4211885} in 2021 considered the corresponding parabolic equation of \eqref{psobo2}, i.e., the following equation
	\begin{align} \label{wangyuzhao}
		\partial_tv = \Delta_pv-|\nabla v|^p+bv.
	\end{align}
	They used the maximum principle to obtain a global Li-Yau type gradient estimate for solutions to equation \eqref{wangyuzhao} on compact Riemannian manifolds.
	
	It is worth mentioning that some mathematicians have studied some similar questions on metric measure spaces. For instance, Coulhon-Jiang-Koskela-Sikora \cite{coulhon2020gradient} gave a gradient estimate for heat kernels on a doubling metric measure space; Zhao-Yang \cite{MR3866881} studied the gradient estimate of the weighted $p$-Laplacian Lichnerowicz equation
	$$\Delta_{p,f}u+cu^{\sigma}=0$$
	on manifolds with $m$-Bakry-Emery Ricci curvature bounded from below. For more studies of this topic, we refer to \cite{MR3509520, MR2350853, MR4498474}.
	
	The most important motivation of the present paper is to study the gradient estimate of the equation \eqref{equ0} in Riemannian manifolds and extend some results obtained for the equation on a domain in an Euclidean space
	$$\Delta_p u + M|\nabla u|^q + |u|^{r-1}u=0  \quad\mbox{in}\quad  \Omega\subset \R^n$$
	by Bidaut-V\'eron, Garcia-Huidobro and V\'eron \cite{MR3261111, bidaut2020priori, bidaut2020DCDS} and Filippucci, Sun and Zheng \cite{filippucci2022priori} to the case on a Riemannian manifold. 
	
	Instead of the Bernstein method and the Keller-Osserman technique, which were used in \cite{MR3261111, filippucci2022priori,bidaut2020priori, bidaut2020DCDS, MR1004713}, we use the Nash-Moser iteration method to approach the gradient estimates inspired by Wang and Zhang \cite{MR2880214} (also see \cite{WW, DiW, Wy}). The Nash-Moser iteration method is known as an integral estimate method and seems to be more suitable for equations defined on Riemannian manifolds. Indeed, when one adopts the Bernstein method to approach such problems, one always needs to employ Barrier functions
	and hence uses the comparison theorems, which may require an assumption on the sectional curvature of the Riemannian manifold (see \cite[Theorem E, Theorem F]{MR3261111}, \cite[Lemma 2.2]{MR2518892}). However, the Nash-Moser iteration method only needs the condition on Ricci curvature.
	\medskip
	
	Now we are in the position to state our results. The following gradient estimate is the main theorem of this paper.
	\begin{thm}\label{t1}
		Let $(M,g)$ be a complete Riemannian manifold with $\mathrm{Ric}_g\geq-(n-1)\kappa g$ for some constant $\kappa\ge 0$. For any solution $u\in C^1(B_R(o))$ of the Hamilton-Jacobi equation \eqref{hamjoco} with $p>1$ and $q>p-1$, which is defined on a geodesic ball $B_R(o)$, we have
		$$\sup_{B_{\frac{R}{2}}(o)}|\nabla u|\leq C_{n,p,q} \left(\frac{1+\sqrt{\kappa} R }{R}\right)^{\frac{1}{q-p+1}}$$
		for some constant $C_{n,p,q}$ depending on $n$, $p$ and $q$.
	\end{thm}
	
	\begin{rem}
		\thmref{t1} can cover and improve some previous results.
		\begin{itemize}
			\item If $q=p$, \thmref{t1} covers the gradient estimate of Cheng and Yau \cite{MR385749} for $p=2$  and Wang and Zhang \cite{MR2880214}  for any $p>1$.
			\item When $p\neq q$, \thmref{t1} improves  some results of  \cite{MR3261111}. It removes the restriction $p\leq n$ in
			\cite[Theorem A]{MR3261111}, and drops the conditions on convexity radius and sectional curvature growth in \cite[Theorem E, Theorem F]{MR3261111}.
		\end{itemize}
	\end{rem}
	
	As a direct consequence, the following Liouville-type theorem is established.
	
	\begin{cor}\label{t2*}
		Let $(M,g)$ be a Riemannian manifold with non-negative Ricci curvature, that is, $\mathrm{Ric}\geq 0$. For any solution $u\in C^1(B_R(o))$ of the Hamilton-Jacobi equation \eqref{hamjoco} defined on $B_R(o)$, if $p>1$ and $q>p-1$,  then there exists a constant $C=C(n,p,q)$ such that
		\begin{align}\label{bvanish}
			\sup_{B_{\frac{R}{2}}(o)}|\nabla u|\leq C(n, p, q)\left(\frac{1}{R}\right)^{\frac{1}{q-p+1}}.
		\end{align}		
		Moreover, if $M$ is a non-compact complete manifold and $u$ is a global solution of equation \eqref{hamjoco} on $M$, then there holds $\nabla u\equiv0$ by letting $R\to\infty$ in \eqref{bvanish}, which means that $u$ is a constant.
	\end{cor}
	
	\begin{rem}
		\corref{t2*} can be viewed as an extension and improvement of \cite[Theorem A]{MR3261111}. If $(M, g)$ is a domain $\Omega$ in $\R^n$, then, obviously $B_d(x)\subset\Omega$ for $d=d(x,\partial\Omega)$ and $\kappa=0$. We can see easily that the estimate in the above \corref{t2*} is a considerable improvement of the corresponding estimate \eqref{jfathma}, i.e.,
		\begin{align*}
			|\nabla u(x)|\leq C(n, p,q)\left(d(x, \partial\Omega)\right)^{-\frac{1}{q-p+1}},
		\end{align*}
		stated in \cite[Theorem A]{MR3261111} and the condition $p\leq n$ supposed in  \cite[Theorem A]{MR3261111} is removed.
	\end{rem}
	
	\begin{thm}\label{addition}
		Assume that $M$ satisfies the same assumptions as in \thmref{t1}. If $u\in C^1(M)$ is a global solution to the Hamilton-Jacobi equation \eqref{hamjoco} on $M$, then for any fixed $a\in M$ and any $x\in M$, we have
		\begin{align}\label{equation:1.11}
			u(a)-c(n,p,q)\kappa^{\frac{1}{2(1-p+q)}}d(x,a)\leq  u(x)\leq u(a)+c(n,p,q)\kappa^{\frac{1}{2(1-p+q)}}d(x,a),
		\end{align}
		where $d(x,a)$ denotes the geodesic distance from $x$ to $a$.
		Especially when $p=q$, equation \eqref{hamjoco} is just the logarithmic transformation of the $p$-harmonic equation. Then for any positive $p$-harmonic function $v$ on $M$, there holds true
		\begin{align}\label{hanack2}
			v(a)e^{-c(n,p.q)\sqrt{\kappa}d(x,a)}\leq v(x)\leq  v(a)e^{c(n,p.q)\sqrt{\kappa}d(x,a)},\quad \forall x\in M.
		\end{align}
	\end{thm}
	
	\begin{rem}
		Inequality \eqref{equation:1.11}  is a logarithmic version of Harnack inequality. The Harnack inequality of $p$-harmonic functions in the case $p =2$ and $\kappa=0$ is due to Cheng and Yau \cite{MR385749}. Kortschwar and Li \cite{MR2518892} obtained a similar estimate but with an assumption on the sectional curvature.  Bidaut-V\'eron, Garcia-Huidobrob and V\'eron \cite[Theorem G, F]{MR3261111} obtained a similar result to \eqref{hanack2}  but with additional conditions on convexity radius and the growth of sectional curvature.
	\end{rem}
	
	When $b(x)$ does not vanish in equation \eqref{equ0}, the gradient estimate for solutions of \eqref{equ0} is more complicated. If $u\in L^k(B_R(o))$ for some $k>2rn$, then we can deduce the following gradient estimate.
	
	\begin{thm}\label{t4}
		Let $(M,g)$ be a complete Riemannian manifold satisfying $\mathrm{Ric}_g\geq-(n-1)\kappa g$ for some constant $\kappa\geq0$. Suppose that $b\in W^{1,\infty}(B_R(o))$ is a real function and the constants $p$, $q$ and $r$ associated with \eqref{equ0} satisfy
		\begin{align}\label{condi0}
			p>1,\quad q>\max\{1,p-1\}\quad\text{and}\quad r\geq 1.
		\end{align}
		If $u\in C^1(B_R(o)) $ is a solution of \eqref{equ0} in a geodesic ball $B_R(o)$ and there exists some $k>2rn$ such that $u\in L^k(B_R(o))$, then
		$$\sup_{B_{\frac{R}{2}}(o)}|\nabla u|\leq\max\left\{1,\, C\left [\left(\frac{1+\sqrt{\kappa}R}{R}\right)^{\frac{1}{q-p+1}}+ \left(R\|b\|^2_{1,\infty}\|u\|^{\frac{2k}{k/r-n}}_{k}\right)^{\frac{1}{q-p+1}}T^{\frac{n }{k/r-n}\frac{1}{q-p+1}} \right]\right\},$$
		where $T  = e^{c_0(1+\sqrt{\kappa} R )}V^{-\frac{2}{n}}R^2$ is the Sobolev constant of Saloff-Coste's Sobolev inequality (see \lemref{salof}) and the constant $C=C(n,p,q,r)$ depends on $n$, $p$, $q$ and $r$.
	\end{thm}
	
	If $u\in L^\infty(\Omega)$, the gradient estimate for solutions to \eqref{equ0} can also be established.
	\begin{thm}\label{t5}
		Let $(M,g)$ be a complete Riemannian manifold with $\mathrm{Ric}_g\geq-(n-1)\kappa g$ for some constant $\kappa\geq0$. Suppose that $b\in W^{1,\infty}(B_R(o))$ is a real function and the constants $p$, $q$ and $r$ in \eqref{equ0} satisfy
		\begin{align}\label{condi0}
			p>1,\quad q>\max\{1,p-1\}\quad\text{and}\quad r\geq 1.
		\end{align}
		Then, for any solution $u\in C^1(B_R(o)) \cap L^{\infty}(B_R(o))$ to \eqref{equ0} in a geodesic ball $B_R(o)$ there holds true
		$$\sup_{B_{\frac{R}{2}}(o)}|\nabla u|\leq\max\left\{1,\, C \left[\left(\frac{1+\sqrt{\kappa}R}{R}\right)^{\frac{1}{q-p+1}}+N\right]\right\}$$
		where the constants
		$$C=C(n,p,q,r)\quad\text{and}\quad  N =\left(\|b\|^2_{1,\infty}+\|b\|^2_{1,\infty}\|u\|^{2r}_{\infty}\right)^{\frac{1}{2(q-p+1)}}.$$
	\end{thm}
	
	\begin{rem}
		Assume that $u$ is an entire solution to equation \eqref{equ0} on a non-compact Riemannian manifold $M$. If $u\in L^\infty$, then $|\nabla u|$ is uniformly bounded by letting $R\to\infty$; however, if $u\in L^k$, then \thmref{t4} can not guarantee that $|\nabla u|$ is globally bounded.
	\end{rem}
	
	The completeness of the Riemannian manifold and the lower boundedness of Ricci curvature are necessary for the application of Saloff-Coste's Sobolev inequalities (\lemref{salof}, see \cite{saloff1992uniformly}). If we consider the equation defined on a region of $\R^n$, then the lower boundedness of the Ricci curvature and the Sobolev inequality holds naturally, thus we can obtain the following result.
	
	\begin{cor}\label{coro}
		Let $\Omega\subset \R^n$ be a domain. Suppose that $u\in C^1(\Omega)$ is a solution to the following equation on $\Omega$
		\begin{align}\label{fsz}
			\Delta_pu+N|\nabla u|^q+|u|^{r-1}u=0
		\end{align}
		with $p$, $q$ and $r$ satisfying \eqref{condi0}. Then, there exists a constant $C=C(n,p,q,r, \|u\|_{L^\infty(\Omega)})$ such that
		$$
		|\nabla u(x)|\leq C\left(1+d(x,\partial\Omega)^{-\frac{1}{q-p+1}}\right), \quad \forall x\in\Omega.
		$$
	\end{cor}
	
	\begin{rem}
		Filippucci-Sun-Zheng \cite{filippucci2022priori} employed the Keller-Osserman method to prove that, if
		$$
		N>0, \quad p>1,\quad r>\max\{1, p-1\} \quad\mbox{and}\quad q>\frac{rp}{r+1},
		$$
		then any positive solution of \eqref{fsz} in a domain $\Omega\subset \R^n$ satisfies
		\begin{align*}
			|\nabla u(x)|<C(n,p,q,r)\left(1+d(x,\partial\Omega)^{-\frac{1}{q-p+1}}\right).
		\end{align*}
		It is worthy to point out that the range of $n$, $p$, $q$ and $r$ are wider in \corref{coro}, but the constant $C(n,p,q,r)$ in the above estimate does not depend on the $L^\infty$ norm of $u$. Even in the cases of Euclidean spaces and positive solutions, the estimate of Filippucci-Sun-Zheng and \corref{coro} cannot cover each other.
	\end{rem}
	
	If we consider positive solutions of equation \eqref{equ0}, we can obtain a local $L^{\infty}$-estimate for such solutions. More precisely, we deduce an $L^{\infty}$-estimate for the solutions of a second-order partial differential inequality.
	
	\begin{prop}\label{c0estimate}
		Let $(M,g)$ be a complete Riemannian manifold and for $1<p<n$ the following Sobolev inequality on $M$ holds true
		\begin{align}\label{lpsobolev}
			\left(\int_M |u|^{\frac{np}{n-p}}\right)^{\frac{n-p}{pn}}\leq C(M)\left(\int_M|u|^p+|\nabla u|^{p}\right)^{\frac{1}{p}} ,\quad \forall u\in W^{1,p}(M).
		\end{align}
		Denote a geodesic ball in $(M,g)$ centered at $o$ and with radius $r$ by $B_R(o)$. If $u\in C^1(B_{R}(o))$ is a positive function which satisfies
		$$\Delta_p u\geq -Cu^r,$$
		where $0< r\leq p-1$, $1< p<n$ and $C$ is a constant, then the following $C^0$-estimate is true
		$$
		\sup_{B_{\frac{R}{2}}(o)}|u|\leq \max\left\{1,\, C(n,p,r,M)\|u\|_{L^p(B_{R}(o))}\right\}.
		$$
	\end{prop}
	
	The gradient estimate in \thmref{t4} requires $u\in  L^k(B_R(o))$ for some $k>2rn$. Combining \lemref{c0estimate} and \thmref{t5}, we can weaken the condition $u\in  L^k(B_R(o))$ to $u\in  L^p(B_R(o))$ for the gradient estimate for positive solutions of equation \eqref{equ0}.
	
	\begin{thm}\label{thm17}
		Let $(M,g)$ be a complete Riemannian manifold satisfying $\mathrm{Ric}_g\geq-(n-1)\kappa g$ for some constant $\kappa\geq 0$. Suppose that the Sobolev inequality \eqref{lpsobolev} holds true on $M$, $u\in C^1(M)$ is a non-negative solution of \eqref{equ0} with $p\geq 2,\ q>p-1$ and $1\leq r\leq p-1$. If $u\in L^p(B_{R}(o))$, then there exist two constants
		$$ N = N\left(n,p,q,r, \|b\|_{W^{1,\infty}(B_{R}(o))}, \|u\|_{L^p(B_{R}(o))}\right)\quad\text{and}\quad
		C=C(n,p,q,r, R)$$
		with $C$ uniformly bounded as $R\to\infty$
		such that
		\begin{align}
			\sup_{B_{\frac{R}{4}}(o)}|\nabla u|\leq\max\left\{1,\, C\left[ \left(\frac{1+\sqrt{\kappa}R}{R}\right)^{\frac{1}{q-p+1}}+N \right] \right\}
		\end{align}
		where $V = {\rm{Vol} }(B_R(o))$ is the volume of the geodesic ball $B_R(o)$.
	\end{thm}
	
	The paper is organized as follows. In Section 2 we give some fundamental definitions and notations used in this paper, establish several important lemmas on $|\nabla u|^2$, and recall the Saloff-Coste Sobolev inequality on a Riemannian manifold. In Section \ref{sect3} we give the proof of \thmref{t1}. We show \thmref{t4} and \thmref{t5} in Section 4 and discuss the gradient estimates on positive solutions in Section 5.

	\section{Preliminaries}
	Throughout this paper, let $(M,g)$ be an $n$-dimensional Riemannian manifold with $\mathrm{Ric}_g\geq -(n-1)\kappa$ for some constant $\kappa\geq0$ and $\nabla$ be the corresponding Levi-Civita connection. For any function $\varphi\in C^1(M)$, we denote $\nabla \varphi\in \Gamma(T^*M)$ by $\nabla \varphi(X)=\nabla_X\varphi$. Denote usually the volume form by $d\vol=\sqrt{\det(g_{ij})}d x_1\wedge\ldots\wedge dx_n$ where $(x_1,\ldots, x_n)$ is a local coordinate chart, and for simplicity, we may omit the volume form in integration over $M$.
	
	For $p>1$, the $p$-Laplace operator is defined by
	$$
	\Delta_p\varphi:=\di(|\nabla \varphi|^{p-2}\nabla \varphi).
	$$
	The solution $\varphi$ of the equation $\Delta_p\varphi=0$ is the critical point of the energy functional
	$$
	E(\varphi)=\int_M|\nabla \varphi|^p.
	$$

	\begin{defn}\label{def1}
		A function $u$ is said to be a (weak) solution of equation \eqref{equ0} if $u\in C^1(M)$ and for all $\psi\in C_0^\infty(M)$, we have
		\begin{align*}
			-\int_M|\nabla u|^{p-2}\la\nabla u,\nabla\psi\ra-\int_M|\nabla u|^q\psi+\int_Mb(x)|u|^{r-1}u\psi=0.
		\end{align*}
	\end{defn}
	It is worth mentioning that any solution $u$ of equation (\ref{equ0}) satisfies $u\in W^{2,2}_{loc}(\Omega)$ and $u\in C^{1,\alpha}(\Omega)$ for some $\alpha\in(0,1)$(for example, see \cite{MR0709038, MR0727034,MR0474389}). Moreover, away from $\{\nabla u=0\}$, $u$ is in fact smooth.
	\medskip
	
	In our proof of gradient estimates of the solution to \eqref{equ0}, the following Sobolev inequality due to Saloff-Coste \cite{saloff1992uniformly} plays an important role.
	\begin{lem}[Saloff-Coste]\label{salof}
		Let $(M,g)$ be a complete manifold with $\mathrm{Ric}\geq-(n-1)\kappa$. For $n>2$, there exists a constant $c_0$ depending only on $n$, such that for all ball $B\subset M$ of radius R and volume $V$ we have
		$$
		\|f\|_{L^{\frac{2n}{n-2}}(B)}^2\leq e^{c_0(1+\sqrt{\kappa}R)}V^{-\frac{2}{n}}R^2\left(\int_B|\nabla f|^2+R^{-2}f^2\right)
		$$
		for any $f\in C^{\infty}_0(B)$. For $n=2$, the above inequality holds with $n$ replaced by any fixed $n'>2$.
	\end{lem}
	\begin{rem}
		For any open region $\Omega\subset M$, if there exists some geodesic ball $B $ such that $\Omega\subset B$, we also have
		$$
		\|f\|_{L^{\frac{2n}{n-2}}(\Omega)}^2\leq e^{c_0(1+\sqrt{\kappa}R)}V^{-\frac{2}{n}}R^2\left(\int_\Omega|\nabla f|^2+R^{-2}f^2\right),
		$$
		for any $f\in W^{1, 2}(\Omega)$. This can be seen from the fact that we choose $\{f_n\}\subset C^\infty_0(\Omega)\subset C^{\infty}_0(B)$ such that  $f_n\to f$  in $W^{1,2}(\Omega)$.
	\end{rem}
	
	For $p>1$, we consider the linearized operator $\Lie$ of $p$-Laplace operator at $u$,
	\begin{align*}
		\Lie(\psi)=\di\left(f^{\frac{p}{2}-1}\nabla \psi+(p-2)f^{\frac{p}{2}-2}\la\nabla \psi,\nabla u\ra\nabla u\right),
	\end{align*}
	where $f = |\nabla u|^2$. We need to prove the following identity.
	
	\begin{lem}\label{lem2.3}
		Let $p>1$ and $r\geq 1$ in \eqref{equ0}. If $u$ is a solution of \eqref{equ0} and $f=|\nabla u|^2$, then the equality
		\begin{align}\label{equ5}
			\begin{split}
				\mathcal L(f)=&\ \left(\frac{p}{2}-1\right)f^{\frac{p}{2}-2} |\nabla f|^2+2f^{\frac{p}{2}-1}(\ric(\nabla u,\nabla u)+|\nabla\nabla u|^2)\\&+qf^{\frac{q}{2}-1}\la\nabla f,\nabla u\ra-2\la\nabla(b|u|^{r-1}u), \nabla u\ra
			\end{split}
		\end{align}
		holds point-wisely in $\{x:f(x)>0\}$.
	\end{lem}
	
	\begin{proof}
		By direct computations we have
		\begin{align*}
			& \di(f^{\frac{p}{2}-1} \nabla f)\\
			=&\ \left(\frac{p}{2}-1\right) f^{\frac{p}{2}-2} |\nabla f|^2 + f^{\frac{p}{2}-1} \Delta f\\
			=&\ \left(\frac{p}{2}-1\right) f^{\frac{p}{2}-2} |\nabla f|^2 + 2f^{\frac{p}{2}-1}\left(\la\nabla\Delta u,\nabla u\ra+\ric(\nabla u,\nabla u)+|\nabla\nabla u|^2\right)
		\end{align*}
		and
		\begin{align*}
			&\di\left((p-2)f^{\frac{p}{2}-2}\la\nabla f,\nabla u\ra\nabla u\right)
			\\
			=&\ (p-2)f^{-1}\la\nabla f,\nabla u\ra\di (f^{\frac{p}{2}-1}\nabla u)+(p-2)\la\nabla(f^{-1}\la\nabla f,\nabla u\ra), f^{\frac{p}{2}-1}\nabla u\ra
			\\
			=&\ (p-2)(f^{\frac{q}{2}}-bu^r)f^{-1}\la\nabla f,\nabla u\ra +2\la\nabla(f^{\frac{q-p}{2}+1}-b(x)|u|^{r-1}uf^{1-\frac{p}{2}}
			-\Delta u), f^{\frac{p}{2}-1}\nabla u\ra
			\\
			=&\ qf^{\frac{q}{2}-1}\la\nabla f,\nabla u\ra-2f^{\frac{p}{2}-1}\la\nabla\Delta u, \nabla u\ra-2\la\nabla(b|u|^{r-1}u), \nabla u\ra.
		\end{align*}
		It is not difficult to obtain \eqref{equ5} from the above identities.
	\end{proof}
	
	Furthermore, we need to show the following point-wise estimate for $\mathcal L(f)$,
	\begin{lem}\label{linear}
		If $u$ is a solution of equation \eqref{equ0} with $p>1$ and $r\geq 1$, then we have the following point-wise estimate on $\{f\neq 0\}$ for $\mathcal L(f)$,
		\begin{align}\label{basiclm}
			\begin{split}
				\mathcal L(f)\geq &-\left|q-\frac{2(p-1)}{n-1}\right|f^{\frac{q-1}{2} }|\nabla f|-2(n-1)\kappa f^{\frac{p}{2}}
				+\frac{2f^{q-\frac{p}{2}+1}  }{n-1}\\
				&-2\la\nabla(b|u|^{r-1}u), \nabla u\ra-\frac{ 4b|u|^{r-1}uf^{\frac{2-p+q}{2}}}{n-1}.
			\end{split}
		\end{align}
	\end{lem}
	
	\begin{proof}
		Let $\{e_1,e_2,\cdots,e_n\}$ be a local orthonormal frame of $TM$ in a domain with $f\neq 0$ such that $e_1=\frac{\nabla u}{|\nabla u|}$. Then we have $u_1 = f^{\frac{1}{2}}$ and
		\begin{align}\label{add}
			u_{11} = \frac{1}{2}f^{-\frac{1}{2}}f_1 = \frac{1}{2}f^{-1}\la\nabla u,\nabla f\ra.
		\end{align}
		The $p$-Laplace operator can be expressed in terms of $f$,
		\begin{align}\label{equ:2.2}
			\Delta_p u
			=&f^{\frac{p}{2}-1}\left((p-1)u_{11}+\sum_{i=2}^nu_{ii}\right).
		\end{align}
		Substituting \eqref{equ:2.2} into equation \eqref{equ0}, we obtain
		\begin{align}\label{equ2}
			(p-1)u_{11}+\sum_{i=2}^nu_{ii}=f^{\frac{q-p}{2}+1}-b|u|^{r-1}uf^{1-\frac{p}{2}}.
		\end{align}
		Using the fact $u_1 = f^{\frac{1}{2}}$ again, we have
		
		\begin{align}\label{equ2.5}
			\frac{|\nabla f|^2}{4f}=\sum_{i=1}^n u_{1i}^2 .
		\end{align}		
		By equality \eqref{equ2.5} and the Cauchy inequality, we estimate the Hessian of $u$ by
		\begin{align}\label{equ2.6}
			|\nabla\nabla u|^2
			\geq& \sum_{i=1}^n u_{1i}^2+\sum_{i=2}^nu_{ii}^2
			\geq \frac{|\nabla f|^2}{4f}+\frac{1}{n-1}\left(\sum_{i=2}^n u_{ii}\right)^2.
		\end{align}
		In view of \eqref{equ2}, we have
		\begin{align}\label{equ2.7}
			\begin{split}
				\frac{1}{n-1}\left(\sum_{i=2}^n u_{ii}\right)^2 = \frac{1}{n-1}
				& \left(
				f^{\frac{q-p}{2}+1}-b|u|^{r-1}uf^{1-\frac{p}{2}}
				-(p-1)u_{11}\right)^2\\
				\geq
				\frac{1}{n-1}
				& \left(
				f^{q-p+2}-2b|u|^{r-1}uf^{2+\frac{q}{2}-p}
				-2(p-1)u_{11}f^{\frac{q-p}{2}+1}
				\right),
			\end{split}
		\end{align}
		where we omit the square term $\left(b|u|^{r-1}uf^{1-\frac{p}{2}} +(p-1)u_{11} \right)^2.$
		Now, substituting \eqref{equ2.6}, \eqref{equ2.7} and Ricci condition $\mathrm{Ric}\geq -(n-1)\kappa g$ into \eqref{equ5}, we have
		\begin{align*}
			\mathcal L(f)
			\geq &\ \frac{p-1}{2} f^{\frac{p}{2}-2} |\nabla f|^2
			-2(n-1)\kappa f^{\frac{p}{2} }+2\left(q-\frac{2(p-1)}{n-1}\right)f^{\frac{q}{2} }u_{11}
			\\
			&-2\la\nabla(b|u|^{r-1}u), \nabla u\ra
			+\frac{2f^{\frac{p}{2}-1} }{n-1} \Big(f^{q-p+2}-2b|u|^{r-1}uf^{2+\frac{q}{2}-p} \Big).
		\end{align*}
		Omitting the non-negative term $\frac{p-1}{2} f^{\frac{p}{2}-2} |\nabla f|^2$ and using \eqref{add}, we deduce from the above inequality that
		\begin{align*}
			\mathcal L(f)
			\geq &
			-2(n-1)\kappa f^{\frac{p}{2} }-\left|q-\frac{2(p-1)}{n-1}\right|f^{\frac{q-1}{2} }|\nabla f|
			-2\la\nabla(b|u|^{r-1}u), \nabla u\ra
			\\
			&
			+\frac{2f^{\frac{p}{2}-1} }{n-1} \left( f^{q-p+2}
			-2b|u|^{r-1}uf^{2+\frac{q}{2}-p} \right)   .
		\end{align*}
		This is the required estimate.
	\end{proof}
	
	\begin{rem}\label{21}
		The assumption $r\geq 1$ in \lemref{lem2.3} and \lemref{linear} can guarantee that there is no singularity due to $\la\nabla(b|u|^{r-1}u), \nabla u\ra$. If $b\equiv0$, then the estimate of $\mathcal L(f)$ is given by
		\begin{align}\label{equa2.9}
			\mathcal L(f)\geq &-\left|q-\frac{2(p-1)}{n-1}\right|f^{\frac{q-1}{2} }|\nabla f|-2(n-1)\kappa f^{\frac{p}{2}}+\frac{2f^{q-\frac{p}{2}+1}  }{n-1},
		\end{align}
		and the assumption $r\geq 1$ is no longer required.
	\end{rem}

	\section{Gradient estimates for solutions to $\Delta_pu-|\nabla u|^q=0$}\label{sect3}
	In this section, we consider the special case $b\equiv0$ of \eqref{equ0}, i.e., the Hamilton-Jacobi equation
	\begin{align}\label{b=0}
		\Delta_pu-|\nabla u|^q=0
	\end{align}
	with $p>1$ and $q>p-1$. We divide the proof of \thmref{t1} into three parts. In the first part, we derive a basic integral inequality of $f=|\nabla u|^2$, which will be used in the second and third parts. In the second part, we give an $L^{\alpha_1}$-estimate of $f$ on a geodesic ball with radius $3R/4$, where $L^{\alpha_1}$ norm of $f$ determines the initial state of the Nash-Moser iteration. Finally, we give a complete proof of our theorem using the Nash-Moser iteration method.
	
	From now on we always assume that $\Omega = B_R(o)\subset M$ is a geodesic ball, and we always use $a_i$ ($i=1,\, 2,\, 3,\, \cdots$) to denote some positive constants depending only on $n$, $p$, $q$ and $r$. We now prove the following integral inequality.
	\begin{lem}\label{31}	
		Let $u\in C^1(\Omega)$ be a solution to equation \eqref{b=0} with $p>1$ and $q>p- 1$, defined on $\Omega\subset M$. Then, there exist constants $a_1 = \min\{1, p-1\}$, $a_2=\left|q-\frac{2(p-1)}{n-1}\right|$ and $a_3$ such that, for any positive number $\alpha$ satisfying
		\begin{align}\label{cond1}
			\frac{a_2^2}{a_1\alpha}\leq\frac{1}{4(n-1)}
		\end{align}
		and any non-negative $\eta\in C^\infty_0(\Omega)$, there holds true
		\begin{align}\label{equa3.3}
			\begin{split}
				&\frac{4a_1\alpha}{(2\alpha+p)^2}\int_{\Omega}\left|\nabla(f^{\frac{\alpha}{2}+\frac{p}{4}}\eta )\right|^2
				+\frac{1}{n-1}\int_{\Omega}f^{\alpha+q-\frac{p}{2}+1}\eta^2
				\\
				\leq&
				\frac{a_3}{\alpha}\int_{\Omega}f^{\frac{p}{2}+\alpha}|\nabla\eta|^2
				+2(n-1)\kappa\int_{\Omega} f^{\frac{p}{2}+\alpha}\eta^2
				.
			\end{split}
		\end{align}
	\end{lem}
	
	\begin{proof}
		By regular theorem, away from $\{f=0\}$, $u$ is smooth. So both sides of \eqref{equa2.9} are in fact smooth. Let $\epsilon>0$ and $\psi = f_\epsilon^{\alpha}\eta^2 $, where $f_\epsilon = (f-\epsilon)^+, \eta\in C^{\infty}_0(B_R(o))$ is non-negative, $\alpha>1$ which will be determined later. Multiply  $\psi $ to both sides of \eqref{equa2.9} and integrate over $\Omega$.
		
		Then direct computation yields
		\begin{align}\label{intineq123}
			\begin{split}
				&\int_{\Omega}\alpha f^{\frac{p}{2} -1}f^{ \alpha-1}_\epsilon\eta^2 \left(|\nabla f|^2+(p-2)f^{-1}\la\nabla u,\nabla f\ra^2\right)
				\\
				&+\int_{\Omega}2f^{\frac{p}{2} }f^{ \alpha-1}_\epsilon\eta \la \nabla f+(p-2)f^{-1}\la\nabla u,\nabla f\ra\nabla u, \nabla\eta\ra
				+\frac{2}{n-1}\int_{\Omega}f^{ q-\frac{p}{2}+1} f_\epsilon^{\alpha}\eta^2
				\\
				\leq&\
				a_2\int_{\Omega}f^{ \frac{q-1}{2} } f_\epsilon^{\alpha}|\nabla f|\eta^2
				+2(n-1)\kappa\int_{\Omega} f^{\frac{p}{2} } f_\epsilon^{\alpha}\eta^2 .
			\end{split}
		\end{align}
		Note that the two terms containing inner products in the inequality \eqref{intineq123} can be controlled as what follows,
		\begin{align*}
			|\nabla f|^2+(p-2)f^{-1}\la\nabla u,\nabla f\ra^2 \geq \min\{1, p-1\}|\nabla f|^2 = a_1 |\nabla f|^2,
		\end{align*}		
		\begin{align*}
			2  \la \nabla f+(p-2)f^{-1}\la\nabla u,\nabla f\ra\nabla u, \nabla\eta\ra\geq -2(p+1)  |\nabla f||\nabla\eta|.
		\end{align*}		
		Hence, in view of of the above two inequalities, we can derive from \eqref{intineq123} by letting $\epsilon\to 0^+$ that,
		\begin{align}\label{intineq2}
			\begin{split}
				&a_1\alpha\int_{\Omega}f^{\frac{p}{2}+\alpha-2}\eta^2 |\nabla f|^2 +\frac{2  }{n-1}\int_{\Omega}f^{\alpha+q-\frac{p}{2}+1}\eta^2\\
				\leq&\ a_2\int_{\Omega} f^{\alpha+\frac{q-1}{2}}|\nabla f|\eta^2 + 2(n-1)\kappa\int_{\Omega} f^{\frac{p}{2}+\alpha}\eta^2 +2(p+1)\int_{\Omega}f^{\frac{p}{2}+\alpha-1}\eta |\nabla f||\nabla\eta|.
			\end{split}
		\end{align}
		Let $R_i$ represent the $i$-th term on the right-hand side of \eqref{intineq2}. By Cauchy inequality, we have
		\begin{align*}
			R_1\leq&\ \frac{a_2^2}{a_1\alpha}\int_{\Omega}f^{\alpha+q-\frac{p}{2}+1}\eta^2 +\frac{a_1\alpha}{4}\int_{\Omega}f^{\frac{p}{2}+\alpha-2}\eta^2 |\nabla f|^2,\\
			R_3\leq &\ \frac{a_1\alpha}{4}\int_{\Omega}f^{\frac{p}{2}+\alpha-2}\eta^2 |\nabla f|^2+\frac{4(p+1)^2}{a_1\alpha}\int_{\Omega}f^{\frac{p}{2}+\alpha}|\nabla\eta|^2.
		\end{align*}
		
		By the above estimates and the choice of $\alpha$ such that
		$$
		\frac{a_2^2}{a_1\alpha}\leq\frac{1}{n-1},
		$$
		we can infer from \eqref{intineq2} that
		\begin{align}\label{equa:3.2}
			\begin{split}
				&\frac{a_1\alpha}{2}\int_{\Omega}f^{\frac{p}{2}+\alpha-2}\eta^2 |\nabla f|^2
				+\frac{1}{n-1}\int_{\Omega}f^{\alpha+q-\frac{p}{2}+1}\eta^2
				\\
				\leq&\
				2(n-1)\kappa\int_{\Omega} f^{\frac{p}{2}+\alpha}\eta^2
				+\frac{4(p+1)^2}{a_1\alpha}\int_{\Omega}f^{\frac{p}{2}+\alpha}|\nabla\eta|^2
				.
			\end{split}
		\end{align}
		
		On the other hand, by using the inequality $(a + b)^2 \leq 2(a^2 +b^2)$ we have
		\begin{align}\label{equa:3.3}
			\int_{\Omega}\left|\nabla\left(f^{\frac{\alpha}{2}+\frac{p}{4}}\eta \right)\right|^2\leq
			\frac{1}{2}\left(\frac{2\alpha+p}{2}\right)^2\int_{\Omega}f^{\alpha+\frac{p}{2}-2}|\nabla f|^2\eta^2
			+2\int_{\Omega}f^{\alpha+\frac{p}{2}}|\nabla\eta|^2 .
		\end{align}
		Immediately, it follows from \eqref{equa:3.2} and \eqref{equa:3.3} that
		\begin{align}\label{han}
			\begin{split}
				&\frac{4a_1\alpha}{(2\alpha+p)^2}\int_{\Omega}\left|\nabla\left(f^{\frac{\alpha}{2}+\frac{p}{4}}\eta \right)\right|^2
				+\frac{1}{n-1}\int_{\Omega}f^{\alpha+q-\frac{p}{2}+1}\eta^2
				\\
				\leq&
				\left(\frac{4(p+1)^2}{a_1\alpha}+\frac{8a_1\alpha}{(2\alpha+p)^2}\right)\int_{\Omega}f^{\frac{p}{2}+\alpha}|\nabla\eta|^2
				+2(n-1)\kappa\int_{\Omega} f^{\frac{p}{2}+\alpha}\eta^2
				.
			\end{split}
		\end{align}
		By choosing a suitable constant $a_3$ depending only on  $p$ such that 
		$$
		 \frac{4(p+1)^2}{a_1\alpha}+\frac{8a_1\alpha}{(2\alpha+p)^2} 
		\leq \frac{a_3}{\alpha},
		$$
		 we finish the proof of \lemref{31}.
	\end{proof}

	\subsection{ An integral inequality}\label{sect3.1}
	\begin{lem}
		Let $(M,g)$ be a complete Riemannian manifold satisfying $\mathrm{Ric}_g\geq-(n-1)\kappa g$ for some constant $\kappa\geq0$. Assume that $u$ is a $C^1$-solution to equation \eqref{b=0} and $f=|\nabla u|^2$. If $p>1$ and $q>p-1$, then there holds true
		\begin{align}\label{44}
			\begin{split}
				&e^{ -\alpha_0}V^{ \frac{2}{n}}\left(\int_{\Omega}f^{\frac{n}{n-2}(\frac{p}{2}+\alpha)}\eta^{\frac{2n}{n-2}}\right)^{\frac{n-2}{n}}
				+a_{6}\alpha R^{2}\int_{\Omega}f^{\alpha+q-\frac{p}{2}+1}\eta^2\\
				\leq& a_{7} \left[R^{2}\int_{\Omega}f^{\frac{p}{2}+\alpha}|\nabla\eta|^2+\alpha_0^2\alpha\int_{\Omega} f^{\frac{p}{2}+\alpha}\eta^2 \right].
			\end{split}
		\end{align}
		Here
		$$\alpha_0=(1+\sqrt{\kappa}R)\max\left\{c_0+1,\, \frac{4(n-1)a_2^2}{a_1 } \right\}.$$
	\end{lem}
	
	\begin{proof}
		By \lemref{salof}, there holds
		\begin{align*}
			\left(\int_{\Omega}f^{\frac{n}{n-2}(\frac{p}{2}+\alpha)}\eta^{\frac{2n}{n-2}}\right)^{\frac{n-2}{n}}
			\leq e^{c_0(1+\sqrt{\kappa} R)}V^{-\frac{2}{n}}\left(R^2\int_{\Omega}\left|\nabla \left(f^{\frac{p}{4}+\frac{\alpha}{2}}\eta\right)\right|^2+\int_{\Omega}f^{\frac{p}{2}+\alpha}\eta^2\right),
		\end{align*}
		where the constant $c_0=c_0(n)$ depends only on $n$. It follows from \eqref{equa3.3} and the above Sobolev inequality that there holds true
		\begin{align}\label{hea1}
			\begin{split}
				&\frac{4a_1\alpha}{(2\alpha+p)^2}e^{-c_0(1+\sqrt{\kappa} R)}V^{\frac{2}{n}}R^{-2}\left(\int_{\Omega}f^{\frac{n}{n-2}(\frac{p}{2}+\alpha)}\eta^{\frac{2n}{n-2}}\right)^{\frac{n-2}{n}}
				+ \frac{1}{n-1}\int_{\Omega}f^{\alpha+q-\frac{p}{2}+1}\eta^2\\
				\leq &\
				\frac{a_3}{\alpha}\int_{\Omega}f^{\frac{p}{2}+\alpha}|\nabla\eta|^2
				+2(n-1) \kappa \int_{\Omega} f^{\frac{p}{2}+\alpha}\eta^2
				+\frac{4a_1\alpha}{(2\alpha+p)^2R^2}\int_{\Omega}   f^{\frac{p}{2}+\alpha}\eta^2,
			\end{split}
		\end{align}
		where we require that $n\neq 2$. We now choose
		$$
		c_1 =\max\left\{c_0+1,\ \frac{4(n-1)a_2^2}{a_1 } \right\}
		$$
		and denote $\alpha_0 = c_1(1+\sqrt{\kappa}R)$. For $\alpha\geq\alpha_0$, there exist constants $a_4$ depending only on $n$, $p$ and $a_5$ depending only on $p$ such that
		$$
		2(n-1) \kappa+\frac{4a_1\alpha}{(2\alpha+p)^2R^2}\leq 2(n-1) \kappa+\frac{a_1}{pR^2}
		\leq \frac{a_4\alpha_0^2}{R^2}
		$$
		and
		$$\frac{a_5}{\alpha}\leq \frac{4a_1\alpha}{(2\alpha+p)^2}.$$
		It follows that
		\begin{align*} &\frac{a_5}{\alpha}e^{-\alpha_0}V^{\frac{2}{n}}R^{-2}\left(\int_{\Omega}f^{\frac{n}{n-2}(\frac{p}{2}+\alpha)}\eta^{\frac{2n}{n-2}}\right)^{\frac{n-2}{n}}
			+ \frac{1}{n-1}\int_{\Omega}f^{\alpha+q-\frac{p}{2}+1}\eta^2 \\
			\leq &\
			\frac{a_3}{\alpha}\int_{\Omega}f^{\frac{p}{2}+\alpha}|\nabla\eta|^2
			+\frac{a_4\alpha_0^2}{R^2} \int_{\Omega} f^{\frac{p}{2}+\alpha}\eta^2 .
		\end{align*}
		By taking some new suitable constants $a_6$ and $a_7$, we obtain the required result.
	\end{proof}

	\subsection{$L^{\alpha_1}$-bound of $f$ in a geodesic ball with radius $3R/4$}\label{sect:3.2}
	\begin{lem}\label{lem42}
		Let $(M,g)$ be a complete Riemannian manifold satisfying $\mathrm{Ric}_g\geq-(n-1)\kappa g$ for some constant $\kappa\geq0$. Assume that $u$ is a $C^1$-solution to equation \eqref{b=0} and $f=|\nabla u|^2$. If $p>1$ and $q>p-1$, then there holds true
		\begin{align*}
			\|f\|_{L^{\alpha_1}(B_{\frac{3R}{4}}(o))}
			\leq
			a_{12}V^{\frac{1}{\alpha_1}}\left(\frac{ 1+\sqrt{\kappa}R}{R}\right)^{\frac{2}{q-p+1}},
		\end{align*}
		where $\alpha_1 = \frac{n}{n-2}\left(\alpha_0+\frac{p}{2}\right)$ and the constant $a_{12}$ depends only on $n$, $p$ and $q$.
	\end{lem}
	
	\begin{proof}
		We set $\alpha=\alpha_0$ in \eqref{44} and obtain
		\begin{align}\label{equation:44}
			\begin{split}
				&e^{ -\alpha_0}V^{ \frac{2}{n}}\left(\int_{\Omega}f^{\frac{n}{n-2}(\frac{p}{2}+\alpha_0)}\eta^{\frac{2n}{n-2}}\right)^{\frac{n-2}{n}}
				+a_{6}\alpha_0 R^{2}\int_{\Omega}f^{\alpha_0+q-\frac{p}{2}+1}\eta^2\\
				&\leq a_{7} \left[R^{2}\int_{\Omega}f^{\frac{p}{2}+\alpha_0}|\nabla\eta|^2+\alpha_0^3\int_{\Omega} f^{\frac{p}{2}+\alpha_0}\eta^2 \right].
			\end{split}
		\end{align}
		If $$f\geq \left(\frac{2a_{7}\alpha_0^2}{a_{6}R^2}\right)^{\frac{1}{q-p+1}},$$
		it is easy to see
		$$
		a_{7} \alpha_0^3 f^{\frac{p}{2}+\alpha_0}\eta^2  \leq \frac{1}{2}
		a_{6}\alpha_0  R^{2} f^{\alpha_0+q-\frac{p}{2}+1}\eta^2.
		$$
		
		Let $\Omega = \Omega_1\cup\Omega_2$, where $$\Omega_1 =\left\{x\in\Omega:f\geq \left(\frac{2a_{7}\alpha_0^2}{a_{6}R^2}\right)^{\frac{1}{q-p+1}}\right\}$$ and $\Omega_2$ is the complement of $\Omega_1$. Then, it follows
		\begin{align}\label{equa:4.5}
			\begin{split}
				a_{7}\alpha_0^3\int_{\Omega} f^{\frac{p}{2}+\alpha_0}\eta^2 =&\ a_{7} \alpha_0^3\int_{\Omega_1} f^{\frac{p}{2}+\alpha_0}\eta^2 + a_{7} \alpha_0^3\int_{\Omega_2} f^{\frac{p}{2}+\alpha_0}\eta^2\\
				\leq&\
				\frac{a_{6}\alpha_0 R^{2}}{2}\int_{\Omega}f^{\alpha_0+q-\frac{p}{2}+1}\eta^2 +\frac{a_{6}\alpha_0 R^{2}}{2} \left(\frac{2a_{7}\alpha_0^2}{a_{6}R^2}\right)^{\frac{\alpha_0+q-\frac{p}{2}+1}{q-p+1}}V,
			\end{split}
		\end{align}
		where $V$ is the volume of $\Omega$.
		
		Choose $\gamma$ such that $0\leq\gamma\leq 1$, $\gamma\in C^\infty_0(B_R(o))$, $\gamma\equiv 1$ in $B_{\frac{3R}{4}}(o)$ and $|\nabla\gamma|\leq\frac{C }{R}$. Let
		$$\eta = \gamma^{\frac{\alpha_0+q-\frac{p}{2}+1}{q-p+1}}.$$
		We take a direct calculation to see that
		\begin{align}\label{equa:4.6}
			a_{7}R^2|\nabla\eta|^2\leq a_{7} C^2 \left(\frac{\alpha_0+q-\frac{p}{2}+1}{q-p+1}\right)^2\eta^{\frac{2\alpha_0+p}{\alpha_0+q-\frac{p}{2}+1}}\leq a_{8}\alpha^2_0\eta^{\frac{2\alpha_0+p}{\alpha_0+q-\frac{p}{2}+1}}.
		\end{align}
		It then follows from \eqref{equa:4.6}, H\"{o}lder inequality and Young inequality that
		\begin{align}\label{equa:4.7}
			\begin{split}
				a_{7}R^2\int_{\Omega} f^{\frac{p}{2}+\alpha_0}|\nabla\eta|^2\leq&\
				a_{8}\alpha_0^2 \int_{\Omega}f^{\frac{p}{2}+\alpha_0}\eta^{\frac{2\alpha_0+p}{\alpha_0+q-\frac{p}{2}+1}}\\
				\leq&\
				a_{8} \alpha_0^2 \left(\int_{\Omega} f^{\alpha_0+q-\frac{p}{2}+1}\eta^2\right)^{\frac{\alpha_0+\frac{p}{2}}{\alpha_0+q-\frac{p}{2}+1}}V^{\frac{q-p+1}{\alpha_0+q-\frac{p}{2}+1}}\\
				\leq&\
				\frac{a_{6}\alpha_0R^2}{2}\int_{\Omega}f^{\frac{p}{2}+\alpha_0+1}\eta^2
				+
				\frac{a_{6}\alpha_0R^2}{2}\left(\frac{2a_{8}\alpha_0}{a_{6}R^2}\right)^{\frac{\alpha_0+q-\frac{p}{2}+1}{q-p+1}} V.
			\end{split}
		\end{align}
		Now, by substituting \eqref{equa:4.5} and \eqref{equa:4.7} into \eqref{44} we obtain the following
		\begin{align}\label{equa:4.8}
			\begin{split}
				&\left(\int_{\Omega}f^{\frac{n}{n-2}(\frac{p}{2}+\alpha_0)}\eta^{\frac{2n}{n-2}}\right)^{\frac{n-2}{n}}
				\\
				\leq &\ e^{ \alpha_0}V^{1-\frac{2}{n}}
				\left[\frac{a_{6}\alpha_0R^2}{2} \left(\frac{2a_{7}\alpha_0^2}{a_{6}R^2}\right)^{\frac{q-\frac{p}{2}+\alpha_0+1}{q-p+1}}
				+ \frac{a_{6}\alpha_0R^2}{2}\left(\frac{2a_{8}\alpha_0}{a_{6}R^2}\right)^{\frac{\alpha_0+q-\frac{p}{2}+1}{q-p+1}}  \right]
				\\
				\leq &\ a_{9}
				e^{ \alpha_0}V^{1-\frac{2}{n}}a_{10}^{\alpha_0}\alpha_0^3\left(\frac{ \alpha_0 }{ R }\right)^{\frac{p+2\alpha_0 }{q-p+1}}.
			\end{split}
		\end{align}
		Taking $1/\left(\alpha_0+\frac{p}{2}\right)$ power of the both sides of \eqref{equa:4.8}, we obtain
		\begin{align}\label{47}
			&\|f\|_{L^{\alpha_1}(B_{\frac{3R}{4}}(o))}\leq a_{11}V^{\frac{1}{\alpha_1}}\left(\frac{\alpha_0}{R}\right)^{\frac{2}{q-p+1}}.
		\end{align}
		Hence, the required inequality follows.
	\end{proof}

	\subsection{Proof of \thmref{t1} and \thmref{addition}}\label{sect3.3}
	\begin{proof}
		By discarding the the second term on the left-hand side in \eqref{44}, we obtain
		\begin{align}\label{410}
			\left(\int_{\Omega}f^{\frac{n}{n-2}(\frac{p}{2}+\alpha)}\eta^{\frac{2n}{n-2}}\right)^{\frac{n-2}{n}}
			\leq &~ a_{7}e^{ \alpha_0}V^{-\frac{2}{n}}\left[R^{2}\int_{\Omega}f^{\frac{p}{2}+\alpha}|\nabla\eta|^2
			+\alpha_0^2\alpha\int_{\Omega} f^{\frac{p}{2}+\alpha}\eta^2 \right].
		\end{align}
		In order to adopt the Nash-Moser iteration, we set
		$$
		\alpha_{l+1}=\frac{n}{n-2}\alpha_l,\quad \Omega_l = B(o,r_l),\quad r_l=\frac{R}{2}+\frac{R}{4^{l }}, \quad l = 1, 2, \ldots
		$$
		and choose $\eta_l\in C^{\infty}_0(\Omega_l)$ such that
		$$	
		\eta_l\equiv 1 \,\text{ in }\,\Omega_{l+1}, \quad 0\leq \eta_l\leq 1 \quad 
		\text{and}\quad |\nabla\eta_l|\leq\frac{C(n)4^l}{R}.
		$$
		We now choose $\alpha$ such that $\alpha+\frac{p}{2} =\alpha_l $ and take $\eta = \eta_l$ in \eqref{410}. By the fact
		\begin{align*}
			\alpha_l\alpha_0^2\eta_l^2+R^2|\nabla\eta_l|^2\leq\alpha_0^2\left(\alpha_0 + \frac{p}{2}\right)\left(\frac{n}{n-2}\right)^{l} +C^2(n)16^l\leq a_{12}^l\alpha_0^2\alpha_1,
		\end{align*}
		we have
		\begin{align*}
			\left(\int_{\Omega_{l+1}}f^{\alpha_{l+1}}\right)^{\frac{1}{\alpha_{l+1}}}
			\leq&\
			\left(a_{7}e^{ \alpha_0}V^{-\frac{2}{n}}\right)^{\frac{1}{\alpha_l}}
			\left(\int_{\Omega_l}\left( \alpha_l\alpha_0^2\eta_l^2+R^2|\nabla\eta_l|^2\right) f^{\alpha_l}\right)^{\frac{1}{\alpha_l}}
			\\
			\leq &\
			\left(a_{7}\alpha_0^2\alpha_1e^{ \alpha_0}V^{-\frac{2}{n}} \right)^{\frac{1}{\alpha_l}} a_{12} ^{\frac{l}{\alpha_l}}
			\left(\int_{\Omega_l}f^{\alpha_l}\right)^{\frac{1}{\alpha_l}}.
		\end{align*}
		By the facts 	
		\begin{align*}
			\sum_{l=1}^\infty\frac{1}{\alpha_l} = \frac{n}{2\alpha_1}=\frac{n-2}{2\alpha_0+p}\quad  \text{ and }\quad \sum_{l=1}^\infty\frac{l}{\alpha_l} = \frac{n^2}{4\alpha_1}=\frac{n(n-2)}{2(2\alpha_0+p)}<  \frac{n(n-2)}{2p},
		\end{align*}
		the following two quantities
		$$
		\left(a_{7}\alpha_0^2\alpha_1e^{ \alpha_0}  \right)^{\sum_{l=1}^{\infty}\frac{1}{\alpha_l}}
		\quad \text{and}\quad
		a_{12} ^{\sum_{l=1}^{\infty}\frac{l}{\alpha_l}}
		$$
		are both uniformly bounded for any $R>0$ and $\kappa\geq 0$. By taking a standard iteration procedure, we have
		\begin{align}\label{equ:iter2}
			\begin{split}
				\|f\|_{L^{\infty}(B_{\frac{R}{2}}(o))}\leq\ & 	\left(a_{7}\alpha_0^2\alpha_1e^{ \alpha_0}V^{-\frac{2}{n}} \right)^{\sum_{l=1}^{\infty}\frac{1}{\alpha_l}} a_{12} ^{\sum_{l=1}^{\infty}\frac{l}{\alpha_l}}\|f\|_{L^{\alpha_1}(B_{\frac{3R}{4}}(o))}\\
				\leq\ &a_{13} V^{-\frac{1}{\alpha_1}}\|f\|_{L^{\alpha_1}(B_{\frac{3R}{4}}(o))}.
			\end{split}
		\end{align}
		
		Combining \eqref{equ:iter2} with \eqref{47} leads to
		\begin{align}\label{4111}
			\|\nabla u\|_{L^{\infty}(B_{\frac{R}{2}}(o))}\leq a_{14} \left(\frac{1+\sqrt{\kappa}R}{R}\right)^{\frac{1}{q-p+1}}.
		\end{align}
		Thus we finish the proof of \thmref{t1}.
	\end{proof}
	
	\begin{cor}
		Let $\Omega\subset\R^n$ be a region and $u\in C^1(\Omega)$ be a solution to the equation
		$$
		\Delta_p u-|\nabla u|^q=0,$$where $p>1$ and $q>p-1$.
		Then we have
		\begin{align*}
			|\nabla u(x)|\leq c(n,p,q)(d(x,\partial\Omega))^{-\frac{1}{q+1-p}}
		\end{align*}
		for any $x\in\Omega$. In addition, if $\Omega=\R^n$, then $u$ is a constant.
	\end{cor}
	
	\begin{proof}
		Denote $R=d(x,\partial\Omega)$. Obviously, we have $B_R(x)\subset \Omega$ and
		$$
		|\nabla u(x)|\leq \sup_{B_R(x)}|\nabla u|\leq a_{16}\left(\frac{1}{R}\right)^{\frac{1}{q-p+1}}
		=a_{16} d(x,\partial\Omega)^{-\frac{1}{q-p+1}}.
		$$
		Here the constant $a_{16}$ depends only on $n$, $p$, and $q$.
	\end{proof}
	
	This corollary is an improvement of {\cite[Theorem A]{MR3261111}}. Another important application of \thmref{t1} is that any entire solution to $\Delta_p u-|\nabla u|^q=0$ on a non-compact complete Riemannian manifold with $\ric_g\geq-(n-1)\kappa g$ satisfies Harnack inequality. \thmref{addition} is just the following
	
	\begin{cor}\label{cor4.4}
		Assume $(M, g)$ satisfies the same assumptions as in \thmref{t1}. Let $u\in C^1(M)$ be a global solution to the equation \eqref{hamjoco} on $M$, i.e.,
		$$\Delta_p u-|\nabla u|^q=0,\quad q>p-1,\quad p>1.$$
		Then, for any fixed $a\in M$ and any $x\in M$ we have
		\begin{align}\label{dist}
			u(a)-c(n,p,q)\kappa^{\frac{1}{2(1-p+q)}}d(x,a)\leq  u(x)\leq u(a)+c(n,p,q)\kappa^{\frac{1}{2(1-p+q)}}d(x,a).
		\end{align}
		Especially, when $p=q$ the equation \eqref{hamjoco} is just the logarithmic transformation of $p$-Laplace equation, then for any positive $p$-harmonic function $v$ on $M$, there holds true
		\begin{align}\label{dist1}
			v(a)e^{-c(n,p.q)\sqrt{\kappa}d(x,a)}
			\leq v(x)\leq  v(a)e^{c(n,p.q)\sqrt{\kappa}d(x,a)},\quad \forall x\in M.
		\end{align}
	\end{cor}
	
	\begin{proof}
		Letting $R\to \infty$ in \eqref{4111},  we have
		\begin{align}\label{413}
			|\nabla u|\leq c(n,p.q)\kappa^{\frac{1}{2(1-p+q)}}.
		\end{align}
		
		Let $d=d(x, a)$ be the distance between $a$ and $x$. For any arc minimizing geodesic segment $\gamma(t):[0,d]\to M$ which connects $a$ and $x$, we have
		\begin{align}\label{415}
			u(x) = u(a)+\int_{0}^d\frac{d}{dt}\left(u\circ\gamma(t) \right)dt.
		\end{align}
		
		Since $|\gamma'|=1$, we infer from  \eqref{413} that
		\begin{align}\label{416}
			\left| \frac{d}{dt}\left(u\circ\gamma(t) \right)\right| \leq |\nabla u(\gamma(t))||\gamma'(t)|\leq c(n,p.q)\kappa^{\frac{1}{2(1-p+q)}}.
		\end{align}
		It follows from \eqref{415} and \eqref{416} that
		\begin{align*}
			|u(x)-u(a)|\leq c(n,p.q)\kappa^{\frac{1}{2(1-p+q)}}d(x,a),
		\end{align*}
		which implies \eqref{dist}.
		Then \eqref{dist1} follows since $u = (1 - p)\ln v$ satisfies
		$$
		\Delta_pu-|\nabla u|^p=0.
		$$
	\end{proof}
	
	L. V\'{e}ron et al. (\cite[Theorem G]{MR3261111}) obtained \eqref{dist1} under the additional condition
	$$
	\begin{cases}
		r_M=\infty, & \mbox{ if } 1<p<2; \\
		\lim_{d(a,x)\to\infty}\frac{|\mathrm{Sec}(x)|}{d(x,a)}=0, & \mbox{ if } p>2,
	\end{cases}
	$$
	where $r_M$ is the convexity radius and $\mathrm{Sec}(x)$ denotes the maximal sectional curvature at $x$. \corref{cor4.4} is an improvement of their results.
	
	\subsection{The case $\dim(M)=n=2$.}
	In the proof of \thmref{t1} above, we have used Saloff-Coste's Sobolev inequality on the embedding $W^{1,2}(B)\hookrightarrow L^{\frac{2n}{n-2}}(B)$ on a manifold. Since the Sobolev exponent $2^*=2n/(n-2)$ requires $n>2$, we did not consider the case $n=2$ in \thmref{t1}. In this subsection, we will explain briefly that \thmref{t1} can also be established when $n=2$.
	
	When $\dim M=n=2$, we need the special case of  Saloff-Coste's Sobolev theorem, i.e., \lemref{salof}. For any $n'> 2$, there holds
	$$
	\|f\|_{L^{\frac{2n'}{n'-2}}(B)}^2\leq e^{c_0(1+\sqrt{\kappa}R)}V^{-\frac{2}{n'}}R^2\left(\int_B|\nabla f|^2+R^{-2}f^2\right)
	$$for any $f\in C^{\infty}_0(B)$.
	
	For example, we can choose $n'=4$, then we have	
	\begin{align*}
		\left(\int_{\Omega}f^{2\alpha+p}\eta^{4}\right)^{2}
		\leq e^{c_0(1+\sqrt{\kappa} R)}V^{-\frac{1}{2}}\left(R^2\int_{\Omega}\left|\nabla \left(f^{\frac{p}{4}+\frac{\alpha}{2}}\eta\right)\right|^2+\int_{\Omega}f^{\frac{p}{2}+\alpha}\eta^2\right).
	\end{align*}
	By the above inequality and \eqref{equa3.3}, we can deduce the following integral inequality by almost the same method as in Section \ref{sect3.1}
	\begin{align} \label{equation3.19}
		\begin{split}
			&e^{ -\alpha_0}V^{ \frac{1}{2}}\left(\int_{\Omega}f^{ p+2\alpha }\eta^{4}\right)^{\frac{1}{2}}
			+a_{6}\alpha R^{2}\int_{\Omega}f^{\alpha+q-\frac{p}{2}+1}\eta^2\\
			&\leq a_{7} \left[R^{2}\int_{\Omega}f^{\frac{p}{2}+\alpha}|\nabla\eta|^2
			+\alpha_0^2\alpha\int_{\Omega} f^{\frac{p}{2}+\alpha}\eta^2 \right],
		\end{split}
	\end{align}
	where $\alpha_0$ is the same as the $\alpha_0$ defined in Section \ref{sect3.1}.
	
	By repeating the same procedure as in Section \ref{sect:3.2}, we can deduce from \eqref{equation3.19} the $L^{\alpha_1}$-bound of $f $ in a geodesic ball with radius $3R/4$
	\begin{align}\label{equation:3.20}
		\|f\|_{L^{\alpha_1}(B_{\frac{3R}{4}}(o))}
		\leq
		a_{12}V^{\frac{1}{\alpha_1}}\left(\frac{ 1+\sqrt{\kappa}R}{R}\right)^{\frac{2}{q-p+1}},
	\end{align}
	where $\alpha_1=p+2\alpha_0$.
	
	For the Nash-Moser iteration, we set $\alpha_l = 2^l(\alpha_0+p/2)$ and $\Omega_l$ by the similar way with that in Section \ref{sect3.3}, and can obtain the following \eqref{equ:iter2}
	\begin{align}\label{equation:3.21}
		\|f\|_{L^{\infty}(B_{\frac{R}{2}}(o))}
		\leq\ &a_{13} V^{-\frac{1}{\alpha_1}}\|f\|_{L^{\alpha_1}(B_{\frac{3R}{4}}(o))}.
	\end{align}
	Combining \eqref{equation:3.20} and \eqref{equation:3.21}, we finally obtain the Cheng-Yau type gradient estimate. Harnack inequality and Liouville type results follow from the Cheng-Yau type gradient estimate.

	\section{Gradient estimate for the solutions of \eqref{equ0}}\label{sect4}
	
	In this section, we consider the case where  $b(x)$ does not vanish. We will provide the proof of \thmref{t4} and \thmref{t5}. The structure of our proof is similar to Section \ref{sect3}. By an abuse of notations, we still use $a_i$, $i=1,\,2\,\cdots$, to denote constant depending on $n$, $p$, $q$ and $r$, but the same $a_i$ may represent different numbers in different sections.

	\begin{lem}\label{41}	
		Let $u\in C^1(\Omega)$ be a solution to equation \eqref{equ0} with $p>1$ and $r\geq 1$, defined on a region $\Omega\subset M$. Then, there exist constants $a_1 = \min\{1, p-1\}$, $a_2=\left|q-\frac{2(p-1)}{n-1}\right|$ and $a_3$ such that, for any positive number $\alpha$ satisfying
		\begin{align}\label{cond41}
			\frac{a_2^2}{a_1\alpha}\leq\frac{1}{4(n-1)}
		\end{align}
		and any non-negative $\eta\in C^\infty_0(\Omega)$, there holds true
		\begin{align}\label{equa42}
			\begin{split}
				&\frac{2a_1\alpha}{(2\alpha+p)^2}\int_{\Omega}\left|\chi\nabla\left(f^{\frac{\alpha}{2}+\frac{p}{4}}\eta\right)\right|^2
				+ \frac{1}{n-1}\int_{\Omega}f^{\alpha+q-\frac{p}{2}+1}\eta^2\chi\\
				\leq &\
				\frac{a_3}{\alpha}\int_{\Omega}f^{\frac{p}{2}+\alpha}|\nabla\eta|^2\chi
				+2(n-1)\left(\kappa+2r^2\|b\|^2_{1,\infty}\right)\int_{\Omega} f^{\frac{p}{2}+\alpha}\eta^2\chi
				\\
				&
				+a_4\|b\|^2_{1,\infty}\int_{\Omega} |u|^{2r}f^{\alpha +\frac{p}{2}}\eta^2\chi,
			\end{split}
		\end{align}
		where $\chi = \chi_{\{f\geq 1\}}$ is the characteristic function of the set $\{x\in\Omega:f(x)\geq 1\}$.
	\end{lem}
	
	\begin{proof}
		By the regularity theory of elliptic equations, away from the set $\{x\in\Omega: f(x)=0\}$, $u$ is smooth enough since the equation is nondegenerate on such a domain. So, the terms appearing on the both sides of \eqref{basiclm} are in fact smooth. Let $\psi = f^{\alpha}\eta^2\chi=(f\chi)^{\alpha}\eta^2$, where $\eta\in C^{\infty}_0(B_R(o))$ is non-negative and $\alpha>1$ which will be determined later. By multiplying the both sides of \eqref{basiclm} by $\psi$ and integrating over $\Omega$, we take integration by part and a direct computation to derive
		\begin{align}\label{intineq44}
			\begin{split}
				&\int_{\Omega}\alpha f^{\frac{p}{2}+\alpha-2}\eta^2\chi \left(|\nabla f|^2+(p-2)f^{-1}\la\nabla u,\nabla f\ra^2\right)
				\\
				&+\int_{\{x\in\Omega:\, f(x)=1\}}|\nabla f|^{-1}f^{\alpha+\frac{p}{2}-1}\eta^2   \left( |\nabla f|^2+(p-2)f^{-1}\la\nabla u,\nabla f\ra^2\right)
				\\
				&+\int_{\Omega}2f^{\frac{p}{2}+\alpha-1}\eta\chi \la \nabla f+(p-2)f^{-1}\la\nabla u,\nabla f\ra\nabla u, \nabla\eta\ra
				+\frac{2}{n-1}\int_{\Omega}f^{\alpha+q-\frac{p}{2}+1}\eta^2\chi
				\\
				\leq&\
				2\int_{\Omega}\la\nabla(b|u|^{r-1}u), \nabla u\ra f^{\alpha}\eta^2\chi
				+ a_2\int_{\Omega}f^{\alpha+\frac{q-1}{2} }|\nabla f|\eta^2\chi \\
				&+\frac{4}{n-1}\int_{\Omega}
				b|u|^{r-1}uf^{\alpha+1+\frac{q-p}{2} } \eta^2\chi
				+2(n-1)\kappa\int_{\Omega} f^{\frac{p}{2}+\alpha}\eta^2\chi,
			\end{split}
		\end{align}
		here we have used Stokes theorem on the region $\{x\in \Omega:f(x)\geq 1\}$ and fact the outward pointing normal vector of $\{x\in \Omega:f(x)\geq 1\}$, denoted by $\nu$, can be expressed by $\nu = -\nabla f /|\nabla f|$.
		
		Three terms containing inner products in the inequality \eqref{intineq44} can be controlled as what follows,
		\begin{align}\label{es1}
			|\nabla f|^2+(p-2)f^{-1}\la\nabla u,\nabla f\ra^2 \geq \min\{1, p-1\}|\nabla f|^2 = a_1 |\nabla f|^2,
		\end{align}		
		\begin{align}\label{es2}
			- 2f^{\frac{p}{2}+\alpha-1} \la \nabla f+(p-2)f^{-1}\la\nabla u,\nabla f\ra\nabla u, \nabla\eta\ra\leq 2(p+1)f^{\frac{p}{2}+\alpha-1} |\nabla f||\nabla\eta|,
		\end{align}	
		and
		\begin{align}\label{es3}
			\la\nabla(b|u|^{r-1}u), \nabla u\ra f^{\alpha} \leq |\nabla b|  |u|^{r} f^{\alpha+\frac{1}{2}}+r|b||u|^{r-1}f^{\alpha+1}.
		\end{align}	
		By \eqref{es1}, we can deduce that the boundary integral term is non-negative, i.e.,
		\begin{align}\label{es4}
			&\int_{\{x\in\Omega:\,f(x)=1\}}|\nabla f|^{-1}f^{\alpha+\frac{p}{2}-1}\eta^2 \left( |\nabla f|^2+(p-2)f^{-1}\la\nabla u,\nabla f\ra^2\right) \\
			\geq & \int_{\{x\in\Omega:\,f(x)=1\}}a_1|\nabla f| f^{\alpha+\frac{p}{2}-1}\eta^2\geq 0.
		\end{align}
		Now, in view of \eqref{es1}, \eqref{es2}, \eqref{es3} and \eqref{es4}, we can derive from \eqref{intineq44} that
		\begin{align}\label{ineq49}
			\begin{split}
				&a_1\alpha\int_{\Omega}f^{\frac{p}{2}+\alpha-2}\eta^2\chi |\nabla f|^2
				+\frac{2  }{n-1}\int_{\Omega}f^{\alpha+q-\frac{p}{2}+1}\eta^2\chi
				\\
				\leq&\ 2\int_{\Omega}|\nabla b|  |u|^{r-1}u f^{\alpha+\frac{1}{2}}\eta^2\chi  + 2r\int_{\Omega} |b| |u|^{r-1}f^{\alpha+1}\eta^2\chi  + a_2\int_{\Omega} f^{\alpha+\frac{q-1}{2}}|\nabla f|\eta^2\chi \\
				&+ 2(n-1)\kappa\int_{\Omega} f^{\frac{p}{2}+\alpha}\eta^2\chi +2(p+1)\int_{\Omega}f^{\frac{p}{2}+\alpha-1}\eta\chi |\nabla f||\nabla\eta|\\
				&+
				\frac{ 4  }{n-1}\int_{\Omega}b|u|^{r}f^{1+\alpha+\frac{q-p}{2}}\eta^2\chi .
			\end{split}
		\end{align}
		Let $R_i$ represent the $i$-th term on the right-hand side of \eqref{ineq49}. By Cauchy inequality, we have
		\begin{align*}
			R_1\leq &\
			\frac{1}{4(n-1)}\int_{\Omega}f^{\alpha+q-\frac{p}{2}+1}\eta^2\chi +
			4(n-1)\int_{\Omega}|\nabla b|^2|u|^{2r}f^{\alpha-q+\frac{p}{2}}\eta^2\chi ,
			\\
			R_2\leq &\
			\frac{1}{4(n-1)}\int_{\Omega}f^{\alpha+q-\frac{p}{2}+1}\eta^2\chi +
			4r^2(n-1)\int_{\Omega}b^2|u|^{2r-2}f^{\alpha-q+\frac{p}{2}+1}\eta^2\chi ,
			\\
			R_3\leq&\
			\frac{a_2^2}{a_1\alpha}\int_{\Omega}f^{\alpha+q-\frac{p}{2}+1}\eta^2\chi +\frac{a_1\alpha}{4}\int_{\Omega}f^{\frac{p}{2}+\alpha-2}\eta^2\chi |\nabla f|^2,
			\\
			R_5\leq &\
			\frac{a_1\alpha}{4}\int_{\Omega}f^{\frac{p}{2}+\alpha-2}\eta^2\chi |\nabla f|^2+\frac{4(p+1)^2}{a_1\alpha}\int_{\Omega}f^{\frac{p}{2}+\alpha}|\nabla\eta|^2\chi ,
		\end{align*}
		and
		\begin{align*}
			R_6\leq &\
			\frac{1}{4(n-1)}\int_{\Omega}f^{\alpha+q-\frac{p}{2}+1}\eta^2\chi +\frac{16}{n-1}\int_{\Omega}b^2|u|^{2r}f^{\alpha-\frac{p}{2}+1}\eta^2\chi .
		\end{align*}
		By picking $\alpha$ such that
		$$
		\frac{a_2^2}{a_1\alpha}\leq\frac{1}{4(n-1)},
		$$
		we can infer from \eqref{ineq49} and the above estimates on $R_i$ that
		\begin{align}\label{equa:4.2}
			\begin{split}
				&\frac{a_1\alpha}{4}\int_{\Omega}f^{\frac{p}{2}+\alpha-2}\eta^2\chi |\nabla f|^2
				+\frac{1}{n-1}\int_{\Omega}f^{\alpha+q-\frac{p}{2}+1}\eta^2\chi
				\\
				\leq&\
				2(n-1)\kappa\int_{\Omega} f^{\frac{p}{2}+\alpha}\eta^2\chi
				+\frac{4(p+1)^2}{a_1\alpha}\int_{\Omega}f^{\frac{p}{2}+\alpha}|\nabla\eta|^2\chi
				+
				\frac{16}{n-1}  \int_{\Omega}f^{\alpha-\frac{p}{2}+1}b^2|u|^{2r}\eta^2\chi
				\\
				&		
				+4(n-1)\int_{\Omega}|\nabla b|^2|u|^{2r}f^{\alpha-q+\frac{p}{2}}\eta^2\chi
				+4r^2(n-1)\int_{\Omega}b^2|u|^{2r-2}f^{\alpha-q+\frac{p}{2}+1}\eta^2\chi .
			\end{split}
		\end{align}
		
		On the other hand, by using the inequality $(a + b)^2 \leq 2(a^2 +b^2)$ we have
		\begin{align}\label{equa:4.3}
			\int_{\Omega}\left|\chi\nabla\left(f^{\frac{\alpha}{2}+\frac{p}{4}}\eta\right)\right|^2\leq
			\frac{1}{2}\left(\frac{2\alpha+p}{2}\right)^2\int_{\Omega}f^{\alpha+\frac{p}{2}-2}|\nabla f|^2\eta^2\chi
			+2\int_{\Omega}f^{\alpha+\frac{p}{2}}|\nabla\eta|^2\chi .
		\end{align}
		It follows from \eqref{equa:4.2} and \eqref{equa:4.3} that
		\begin{align}\label{equa412}
			\begin{split}
				&\frac{2a_1\alpha}{(2\alpha+p)^2}\int_{\Omega}\left|\chi\nabla\left(f^{\frac{\alpha}{2}+\frac{p}{4}}\eta\right)\right|^2
				+\frac{1}{n-1}\int_{\Omega}f^{\alpha+q-\frac{p}{2}+1}\eta^2\chi
				\\
				\leq&
				\left(\frac{4(p+1)^2}{a_1\alpha}+\frac{4a_1\alpha}{(2\alpha+p)^2}\right)\int_{\Omega}f^{\frac{p}{2}+\alpha}|\nabla\eta|^2\chi
				+2(n-1)\kappa\int_{\Omega} f^{\frac{p}{2}+\alpha}\eta^2\chi
				\\
				&
				+
				\frac{16}{n-1}
				\int_{\Omega}f^{\alpha-\frac{p}{2}+1}b^2|u|^{2r}\eta^2\chi+4(n-1)\int_{\Omega}|\nabla b|^2|u|^{2r}f^{\alpha-q+\frac{p}{2}}\eta^2\chi \\
				&+4r^2(n-1)\int_{\Omega}b^2|u|^{2r-2}f^{\alpha-q+\frac{p}{2}+1}\eta^2\chi .
			\end{split}
		\end{align}
		
		By the definition of $\chi$, we know that
		$$(\chi f)^{\xi}\leq (\chi f)^{\delta}$$
		as long as $0<\xi\leq\delta$. Using the conditions $p>1$ and $q\geq 1$, we have
		$$
		\alpha-\frac{p}{2}+1\leq \alpha+\frac{p}{2},\quad  \alpha-q+\frac{p}{2}<  \alpha+\frac{p}{2},\quad  \alpha-q+\frac{p}{2}+1\leq  \alpha+\frac{p}{2}.
		$$
		Although each of the last three terms on the right-hand side of \eqref{equa412} is of different powers of $f$, we can raise their powers up to $p/2 + \alpha$. Hence
		\begin{align}\label{he0}
			\begin{split}
				&\frac{2a_1\alpha}{(2\alpha+p)^2}\int_{\Omega}\left|\chi\nabla\left(f^{\frac{\alpha}{2}+\frac{p}{4}}\eta\right)\right|^2
				+\frac{1}{n-1}\int_{\Omega}f^{\alpha+q-\frac{p}{2}+1}\eta^2\chi
				\\
				\leq&\
				\left(\frac{4(p+1)^2}{a_1\alpha}+\frac{4a_1\alpha}{(2\alpha+p)^2}\right)\int_{\Omega}f^{\frac{p}{2}+\alpha}|\nabla\eta|^2\chi				
				+4r^2(n-1)\|b\|^2_{1,\infty}\int_{\Omega} |u|^{2r-2}f^{\alpha+\frac{p}{2}}\eta^2\chi
				\\
				&
				+
				\left(\frac{16}{n-1}+4(n-1)\right)\|b\|^2_{1,\infty}
				\int_{\Omega}f^{\alpha+\frac{p}{2}} |u|^{2r}\eta^2\chi +2(n-1)\kappa\int_{\Omega} f^{\frac{p}{2}+\alpha}\eta^2\chi.
			\end{split}
		\end{align}
		According to Young inequality, we have
		\begin{align}
			|u|^{2r-2}\leq \frac{1}{r}+\frac{r-1}{r}|u|^{2r}\leq 1+|u|^{2r},\quad  \forall  r\geq 1.
		\end{align}
		Hence we can infer from \eqref{he0} that
		\begin{align}\label{he}
			\begin{split}
				&\frac{2a_1\alpha}{(2\alpha+p)^2}\int_{\Omega}\left|\chi\nabla\left(f^{\frac{\alpha}{2}+\frac{p}{4}}\eta\right)\right|^2
				+ \frac{1}{n-1}\int_{\Omega}f^{\alpha+q-\frac{p}{2}+1}\eta^2\chi\\
				\leq &\
				\frac{a_3}{\alpha}\int_{\Omega}f^{\frac{p}{2}+\alpha}|\nabla\eta|^2\chi
				+2(n-1)\left(\kappa+2r^2\|b\|^2_{1,\infty}\right)\int_{\Omega} f^{\frac{p}{2}+\alpha}\eta^2\chi
				\\
				&
				+a_4\|b\|^2_{1,\infty}\int_{\Omega} |u|^{2r}f^{\alpha +\frac{p}{2}}\eta^2\chi,
			\end{split}
		\end{align}
		where we choose a suitable constant $a_3$ depending only on $n,p$ and $q$ such that
		$$
		\frac{4(p+1)^2}{a_1\alpha}+\frac{4a_1\alpha}{(2\alpha+p)^2} \leq \frac{a_3}{\alpha},
		$$
		and
		$$
		a_4 = \frac{16}{n-1}+4(n-1)+4r^2(n-1).
		$$
	\end{proof}

	\begin{lem}
		Assume the same conditions as in \lemref{31} are satisfied and additionally $q\geq 1$. Then, there exist constants $a_4$, $a_5$ and $a_6$, which depend only on $n$, $p$, $q$ and $r$, such that 
		\begin{align*}
			\begin{split}
				&\frac{2a_5}{\alpha}e^{-c_0(1+\sqrt{\kappa} R)}V^{\frac{2}{n}}R^{-2}\left(\int_{\Omega}(\chi f)^{\frac{n}{n-2} (\frac{p}{2}+\alpha)}\eta^{\frac{2n}{n-2}}\right)^{\frac{n-2}{n}}+\frac{1}{n-1}\int_{\Omega}(\chi f)^{\alpha+q-\frac{p}{2}+1}\eta^2\\
				\leq&\
				a_6\left(\frac{(1+\sqrt{\kappa}R)^2}{R^{2}}+B\right)\int_{\Omega} (\chi f)^{\frac{p}{2}+\alpha}\eta^2 +\frac{a_3}{\alpha}\int_M(\chi f)^{\frac{p}{2}+\alpha}|\nabla\eta|^2 \\
				&
				+a_4 B\int_M|u|^{2r}f^{\alpha+\frac{p}{2}}\eta^2\chi,
			\end{split}
		\end{align*}
		where $B = \|b\|^2_{1,\infty}.$
	\end{lem}
	
	\begin{proof}
		By \lemref{salof}, there holds
		\begin{align*}
			\left(\int_{\Omega_1}f^{\frac{n}{n-2}(\frac{p}{2}+\alpha)}\eta^{\frac{2n}{n-2}}\right)^{\frac{n-2}{n}}
			\leq e^{c_0(1+\sqrt{\kappa} R)}V^{-\frac{2}{n}}\left(R^2\int_{\Omega_1}\left| \nabla \left(  f^{\frac{p}{4}+\frac{\alpha}{2}}\eta\right)\right|^2+\int_{\Omega_1}f^{\frac{p}{2}+\alpha}\eta^2\right),
		\end{align*}
		where $\Omega_1=\{x\in\Omega:f(x)> 1\}$. Since $\chi^t=\chi, \forall t>0$, the above inequality can be re-expressed as
		\begin{align*}
			\left(\int_{\Omega}(\chi f)^{\frac{n}{n-2}(\frac{p}{2}+\alpha)}\eta^{\frac{2n}{n-2}}\right)^{\frac{n-2}{n}}
			\leq e^{c_0(1+\sqrt{\kappa} R)}V^{-\frac{2}{n}}\left(R^2\int_{\Omega}\left|\chi\nabla \left(  f^{\frac{p}{4}+\frac{\alpha}{2}}\eta\right)\right|^2+\int_{\Omega}(\chi f)^{\frac{p}{2}+\alpha}\eta^2\right).
		\end{align*}
		It follows from \eqref{he} and the above  inequality that
		\begin{align}\label{he32}
			\begin{split}
				&\frac{2a_1\alpha}{(2\alpha+p)^2}e^{-c_0(1+\sqrt{\kappa} R)}V^{\frac{2}{n}}R^{-2}\left(\int_{\Omega}(\chi f)^{\frac{n}{n-2}(\frac{p}{2}+\alpha)}\eta^{\frac{2n}{n-2}}\right)^{\frac{n-2}{n}}
				+ \frac{1}{n-1}\int_{\Omega}f^{\alpha+q-\frac{p}{2}+1}\eta^2\chi\\
				\leq &\
				\frac{a_3}{\alpha}\int_{\Omega}f^{\frac{p}{2}+\alpha}|\nabla\eta|^2\chi
				+2(n-1)\left(\kappa+2r^2\|b\|^2_{1,\infty}\right)\int_{\Omega} f^{\frac{p}{2}+\alpha}\eta^2\chi
				\\
				&
				+a_4\|b\|^2_{1,\infty}\int_{\Omega} |u|^{2r}f^{\alpha +\frac{p}{2}}\eta^2\chi
				+\frac{2a_1\alpha}{(2\alpha+p)^2R^2}\int_{\Omega}(\chi f)^{\frac{p}{2}+\alpha}\eta^2.
			\end{split}
		\end{align}
		We choose suitable $a_5$ and $a_6$ depending only on $n$, $p$, $q$ and $r$ such that
		\begin{align*}
			\frac{2a_5}{\alpha}\leq \frac{2a_1\alpha}{(2\alpha+p)^2}
			\qquad \mbox{and} \qquad
			a_6\left(\frac{(1+\sqrt{\kappa}R)^2}{R^{2}}+B\right)\geq 2(n-1)(\kappa+2r^2B)+\frac{2a_1\alpha}{(2\alpha+p)^2}R^{-2},
		\end{align*}
		then it follows that
		\begin{align}\label{equ32}\begin{split}
				&\frac{2a_5}{\alpha}e^{-c_0(1+\sqrt{\kappa} R)}V^{\frac{2}{n}}R^{-2}\left(\int_{\Omega}(\chi f)^{\frac{n}{n-2} (\frac{p}{2}+\alpha)}\eta^{\frac{2n}{n-2}}\right)^{\frac{n-2}{n}}+\frac{1}{n-1}\int_{\Omega}(\chi f)^{\alpha+q-\frac{p}{2}+1}\eta^2\\
				\leq&\
				a_6\left(\frac{(1+\sqrt{\kappa}R)^2}{R^{2}}+B\right)\int_{\Omega} (\chi f)^{\frac{p}{2}+\alpha}\eta^2 +\frac{a_3}{\alpha}\int_M(\chi f)^{\frac{p}{2}+\alpha}|\nabla\eta|^2\\
				&+a_4 B\int_M|u|^{2r}f^{\alpha+\frac{p}{2}}\eta^2\chi.
			\end{split}
		\end{align}
		Thus, we finish the proof of the lemma.
	\end{proof}
	
	\subsection{The case of $u\in L^k(B_R(o))$}
	To deal with the terms involved with $u$, i.e., the last term on the right-hand side of \eqref{equ32}, we use a technique introduced by Wang-Wang(\cite{wang2021gradient1}). The key point is that we use the first term on the left-hand side of \eqref{equ32} to absorb the high-order term resulting from the last term on the right-hand side of \eqref{equ32}.
	
	\begin{lem}\label{lem4.2}
		Assume $u$ is a solution to equation \eqref{equ0} in $B_R(o)$ with $p>1, q\geq 1$ and $r\geq 1$ and $u\in C^1(B_R(o))\cap L^k(B_R(o))$ where $k\geq 2rn$. Let $f$ and $\chi$ be the same as in the above lemma. Let $\alpha_0$ be the unique positive solution of the equation
		$$
		\alpha^2_0=c^2_1\left( 1+\sqrt{\kappa} R \right)^2+R^2 \|b\|^2_{1,\infty}\|u\|^{\frac{2k}{k/r-n}}_{k}(T\alpha_0)^{\frac{n }{k/r-n}},
		$$
		where the constant $c_1$ depends only on $n,p,q$ (see \eqref{alpha0cond}) and $T=e^{c_0(1+\sqrt{\kappa}R)}V^{-\frac{2}{n}}R^2$. Then, there holds true
		\begin{align*}
			\begin{split}
				&e^{-\alpha_0}V^{ \frac{2}{n}}\left(\int_{\Omega}(\chi f)^{\frac{n}{n-2}(\frac{p}{2}+\alpha)}\eta^{\frac{2n}{n-2}}\right)^{\frac{n-2}{n}}
				+a_{8}\alpha  R^{2}\int_{\Omega}(\chi f)^{\alpha+q-\frac{p}{2}+1}\eta^2\\
				\leq &\ a_{9}\left(\alpha_0^2\left(\frac{\alpha}{\alpha_0}\right)^{\frac{n}{2c-n}}\alpha\int_{\Omega} (\chi f)^{\frac{p}{2}+\alpha}\eta^2
				+ R^{2}\int_{\Omega}(\chi f)^{\frac{p}{2}+\alpha}|\nabla\eta|^2\right).
			\end{split}
		\end{align*}
	\end{lem}	
	
	\begin{proof}
		For the third term on the right-hand side of \eqref{equ32}, we choose $\mu$, $\nu\in(1,\infty)$ such that
		$$\frac{1}{\mu}+\frac{1}{\nu}=1,\quad\quad  \frac{1}{\mu}+ \frac{n}{\nu(n-2)}=\frac{c}{c-1},$$
		where $c = \frac{k}{2r}>n$. Then, by H\"older and Young inequalities we have
		\begin{align}
			\label{equ10}
			\begin{split}
				\int_M |u|^{2r}(\chi f)^{\alpha+\frac{p}{2}}\eta^2
				\leq &\ \| u^{2r}\|_c\left(\int_{\Omega}\left(\eta^2 (\chi f)^{\alpha+\frac{p}{2}}\right)^{\frac{c}{c-1}}\right)^{\frac{c-1}{c}}\\
				\leq &\
				\|u\|^{2r}_{k}\left(\int_{\Omega}\eta^2 (\chi f)^{\alpha+\frac{p}{2}} \right)^{\frac{c-1}{c\mu}}
				\left(\int_{\Omega}\left(\eta^2 (\chi f)^{\alpha+\frac{p}{2}}\right)^{\frac{n}{n-2}}\right)^{\frac{c-1}{c\nu}}\\
				\leq&\
				\| u \|^{2r}_{k}\left[\epsilon^{-\frac{n }{2c-n}}\int_{\Omega}\eta^2 (\chi f)^{\alpha+\frac{p}{2}}
				+\epsilon\left(\int_{\Omega}\left(\eta^2 (\chi f)^{\alpha+\frac{p}{2}}\right)^{\frac{n}{n-2}}\right)^{\frac{n-2}{n}}\right].
			\end{split}
		\end{align}

		Now we choose $\epsilon$ small enough such that
		\begin{align}
			\label{ep1}
			\begin{split}
				a_4\|b\|^2_{1,\infty}\| u \|^{2r}_{k}\epsilon =\frac{a_5}{\alpha}e^{-c_0(1+\sqrt{\kappa} R)}V^{\frac{2}{n}}R^{-2} .
			\end{split}
		\end{align}
		Set $$T=T(R)=e^{c_0(1+\sqrt{\kappa} R)}V^{-\frac{2}{n}}R^{2},$$
		then we have
		$$
		\epsilon^{-\frac{n }{2c-n}}=  \left( \frac{a_4}{a_5}T\alpha \| u \|^{2r}_{k} \right)^{\frac{n }{2c-n}}.
		$$
		Now, let $\epsilon$ satisfy \eqref{ep1}, by substituting \eqref{equ10} into \eqref{equ32} we obtain
		\begin{align}
			\label{equa22}
			\begin{split}
				&\frac{a_5}{\alpha}e^{-c_0(1+\sqrt{\kappa}R)}V^{\frac{2}{n}}R^{-2}\left(\int_{\Omega}(\chi f)^{\frac{n}{n-2}(\frac{p}{2}+\alpha)}\eta^{\frac{2n}{n-2}}\right)^{\frac{n-2}{n}}
				+ \frac{1}{n-1}\int_{\Omega}(\chi f)^{\alpha+q-\frac{p}{2}+1}\eta^2\\
				\leq&\
				\left(a_6\frac{(1+\sqrt{\kappa}R)^2}{R^{2}}
				+a_4\|b\|^2_{1,\infty}\|u\|^{\frac{4rc}{2c-n}}_{k} (T\alpha)^{\frac{n }{2c-n}}\right)
				\int_{\Omega} (\chi f)^{\frac{p}{2}+\alpha}\eta^2 +\frac{a_3}{\alpha}\int_M(\chi f)^{\frac{p}{2}+\alpha}|\nabla\eta|^2\\
				\leq&\
				a_7 \left(\frac{(1+\sqrt{\kappa}R)^2}{R^{2}}
				+ \|b\|^2_{1,\infty}\|u\|^{\frac{4rc}{2c-n}}_{k}(T\alpha)^{\frac{n }{2c-n}}\right) \int_{\Omega} (\chi f)^{\frac{p}{2}+\alpha}\eta^2 +\frac{a_3}{\alpha}\int_{\Omega}(\chi f)^{\frac{p}{2}+\alpha}|\nabla\eta|^2,
			\end{split}
		\end{align}
		where $a_7>\max\{a_6,a_4\}$ depends only on $n,p,q$ and $r$. Now we define $c_1$ by
		\begin{align}\label{alpha0cond}
			c_1=\max\left\{\frac{4(n-1)a_2^2}{a_1 },\ 1+c_0\right\}.
		\end{align}
		Since $k\geq 2rn$ we have $\frac{n }{ 2c-n  }\leq 1$, thus the following equation
		\begin{align}\label{alpha0}
			\alpha^2_0=c^2_1\left( 1+\sqrt{\kappa} R \right)^2+R^2 \|b\|^2_{1,\infty}\|u\|^{\frac{4rc}{2c-n}}_{k}(T\alpha_0)^{\frac{n }{2c-n}}
		\end{align}
		has a unique positive solution $\alpha_0$. Then, it is easy to see that
		$$
		\alpha_0 \geq c_1(1+\sqrt{\kappa}R)\geq 1.
		$$
		It follows from $\alpha_0\geq 1$ and $\frac{n}{2c-n}\leq 1$ that
		\begin{align}\label{alpha0in1}
			\alpha^2_0\leq c^2_1\left( 1+\sqrt{\kappa} R \right)^2+R^2 \|b\|^2_{1,\infty}\|u\|^{\frac{4rc}{2c-n}}_{k}T^{\frac{n }{2c-n}} \alpha_0.
		\end{align}
		Consequently, one can see from the quadratic inequality \eqref{alpha0in1} about $\alpha_0$ that
		\begin{align}\label{alpha0u}
			\alpha_0
			\leq\ &c_1\left(1+\sqrt{\kappa}R\right)+R^2\|b\|^2_{1,\infty}\|u\|^{\frac{4rc}{2c-n}}_{k}T^{\frac{n }{2c-n}}.
		\end{align}
		For any $\alpha\geq\alpha_0$, it follows from \eqref{alpha0} that
		\begin{align*}
			\frac{(1+\sqrt{\kappa}R)^2}{R^{2}}
			+ \|b\|^2_{1,\infty}\|u\|^{\frac{4rc}{2c-n}}_{k}(T\alpha)^{\frac{n }{2c-n}}\leq \frac{\alpha_0^2}{R^2}\left(\frac{\alpha}{\alpha_0}\right)^{\frac{n}{2c-n}}.
		\end{align*}
		Hence, it follows from \eqref{equa22} that
		\begin{align}
			\begin{split}
				&\frac{a_5}{\alpha}e^{-\alpha_0}V^{\frac{2}{n}}R^{-2}\left(\int_{\Omega}(\chi f)^{\frac{n}{n-2}(\frac{p}{2}+\alpha)}\eta^{\frac{2n}{n-2}}\right)^{\frac{n-2}{n}}+\frac{1}{n-1}\int_{\Omega}(\chi f)^{\alpha+q-\frac{p}{2}+1}\eta^2
				\\
				\leq&\
				a_7  \frac{\alpha_0^2}{R^{2}}
				\left(\frac{\alpha}{\alpha_0}\right)^{\frac{n}{2c-n}}\int_{\Omega} (\chi f)^{\frac{p}{2}+\alpha}\eta^2 +\frac{a_3}{\alpha}\int_{\Omega}(\chi f)^{\frac{p}{2}+\alpha}|\nabla\eta|^2.
			\end{split}
		\end{align}
		Immediately, we can conclude that there holds true
		\begin{align}\label{equa3.14}
			\begin{split}
				&e^{-\alpha_0}V^{ \frac{2}{n}}\left(\int_{\Omega}(\chi f)^{\frac{n}{n-2}(\frac{p}{2}+\alpha)}\eta^{\frac{2n}{n-2}}\right)^{\frac{n-2}{n}}
				+a_{8}\alpha  R^{2}\int_{\Omega}(\chi f)^{\alpha+q-\frac{p}{2}+1}\eta^2
				\\
				\leq &\ a_{9}\left(
				\alpha_0^2\left(\frac{\alpha}{\alpha_0}\right)^{\frac{n}{2c-n}}\alpha\int_{\Omega} (\chi f)^{\frac{p}{2}+\alpha}\eta^2
				+
				R^{2}\int_{\Omega}(\chi f)^{\frac{p}{2}+\alpha}|\nabla\eta|^2
				\right).
			\end{split}
		\end{align}
		The proof of this lemma is finished.
	\end{proof}
	
	\subsubsection{$L^{\alpha_1}$ bound of gradient in a geodesic ball with radius $3R/4$}
	We first prove the following lemma.
	
	\begin{lem}\label{lpbound41}
		Let $(M,g)$ be a complete Riemannian manifold satisfying $\mathrm{Ric}_g\geq-(n-1)\kappa g$ for some constant $\kappa\geq0$. Assume that $u$ and $f$ are the same as in  \lemref{lem4.2}. If $b\in W^{1,\infty}(B_R(o))$ is a real function and the constants $p$, $q$ and $r$ associated with \eqref{equ0} satisfy
		\begin{align*}
			p>1,\quad q>\max\{1,p-1\}\quad\text{and}\quad r\geq 1.
		\end{align*}
		then there exist $\alpha_1 = \left(\alpha_0+\frac{p}{2}\right)\frac{n}{n-2}$ and a constant $a_{12}$, which depends only on $n,p,q$ and $r$, such that
		\begin{align}
			\|\chi f\|_{L^{\alpha_1}(B_{\frac{3R}{4}}(o))}\leq a_{12}V^{\frac{1}{\alpha_1}}\left(\frac{\alpha^2_0}{R^2}\right)^{\frac{1}{q-p+1}}.
		\end{align}
		Here $\alpha_0$ is determined in the above \lemref{lem4.2}.
	\end{lem}
	
	\begin{proof}
		We choose $\alpha=\alpha_0$ in \eqref{equa3.14} and we obtain
		\begin{align}\label{equation:4.14}
			\begin{split}
				&e^{-\alpha_0}V^{ \frac{2}{n}}\left(\int_{\Omega}(\chi f)^{\frac{n}{n-2}(\frac{p}{2}+\alpha_0)}\eta^{\frac{2n}{n-2}}\right)^{\frac{n-2}{n}}
				+a_{8}\alpha_0  R^{2}\int_{\Omega}(\chi f)^{\alpha_0+q-\frac{p}{2}+1}\eta^2
				\\
				\leq &\ a_{9}\left(
				\alpha_0^{3} \int_{\Omega} (\chi f)^{\frac{p}{2}+\alpha_0}\eta^2
				+
				R^{2}\int_{\Omega}(\chi f)^{\frac{p}{2}+\alpha_0}|\nabla\eta|^2
				\right).
			\end{split}
		\end{align}
		When
		$$
		f\geq \left(\frac{2a_{9} \alpha_0^{2}}{a_{8}R^2}\right)^{\frac{1}{q-p+1}},
		$$
		obviously we have
		$$
		a_9 \alpha_0^{3}(\chi f)^{\frac{p}{2}+\alpha_0}\eta^2\leq \frac{1}{2}a_{8}\alpha_0  R^{2} (\chi f)^{\alpha_0+q-\frac{p}{2}+1}\eta^2.
		$$
		Now, we decompose $\Omega$ into $\Omega=\Omega_1\cup\Omega_2$, where
		$$\Omega_1 = \left\{f\geq \left(\frac{2a_{9} \alpha_0^{2}}{a_{8}R^2}\right)^{\frac{1}{q-p+1}}\right\}
		$$
		and $\Omega_2$ is the complement of $\Omega_1$, then we have
		\begin{align}\label{equation:4.15}
			a_{9}
			\alpha_0^{3} \int_{\Omega_1} (\chi f)^{\frac{p}{2}+\alpha_0}\eta^2
			\leq &\ \frac{a_{8}\alpha_0  R^{2}}{2}
			\int_{\Omega}(\chi f)^{\alpha_0+q-\frac{p}{2}+1}\eta^2
		\end{align}
		and
		\begin{align}\label{equation:4.16}
			a_{9}
			\alpha_0^{3} \int_{\Omega_2} (\chi f)^{\frac{p}{2}+\alpha_0}\eta^2
			\leq &\ a_{9}
			\alpha_0^{3} V\left(\frac{2a_{9} \alpha_0^{2}}{a_{8}R^2}\right)^{\frac{\alpha_0+\frac{p}{2}}{q-p+1}},
		\end{align}		
		where $V$ is the volume of $\Omega$. Plugging \eqref{equation:4.15} and \eqref{equation:4.16} into \eqref{equation:4.14} leads to
		\begin{align}\label{equa24}
			\begin{split}
				&e^{-\alpha_0}V^{ \frac{2}{n}}\left(\int_{\Omega}(\chi f)^{\frac{n}{n-2}(\frac{p}{2}+\alpha_0)}\eta^{\frac{2n}{n-2}}\right)^{\frac{n-2}{n}}
				+\frac{1}{2}a_{8}\alpha_0  R^{2}\int_{\Omega}(\chi f)^{\alpha_0+q-\frac{p}{2}+1}\eta^2
				\\
				\leq &a_{9} R^{2}\int_{\Omega}(\chi f)^{\frac{p}{2}+\alpha_0}|\nabla\eta|^2
				+\frac{1}{2}a_{8}\alpha_0 R^2V \left(\frac{2a_{9}\alpha_0^2}{a_{8}R^2}\right)^{\frac{\alpha_0+q-\frac{p}{2}+1}{q-p+1}}.
			\end{split}
		\end{align}
		
		Let the cut-off function $\gamma\in C^{\infty}_0(B_R(o))$ satisfy
		$$
		\begin{cases}
			0\leq\gamma\leq 1, \quad
			\gamma\equiv 1\text{ in }  B_{\frac{3R}{4}}(o);\\
			|\nabla\gamma|\leq\frac{C }{R},
		\end{cases}
		$$
		and
		$$\eta = \gamma^{\frac{\alpha_0+q-\frac{p}{2}+1}{q-p+1}}.$$
		Then we have
		\begin{align}\label{314}
			a_{9}R^2|\nabla\eta|^2
			\leq
			a_{9} C^2 \left(\frac{\alpha_0+q-\frac{p}{2}+1}{q-p+1}\right)^2\eta^{\frac{2\alpha_0+p}{\alpha_0+q-\frac{p}{2}+1}}
			\leq
			a_{10}\alpha^2_0\eta^{\frac{2\alpha_0+p}{\alpha_0+q-\frac{p}{2}+1}}.
		\end{align}
		By H\"older  and Young inequalities, we obtain
		\begin{align}\label{315}
			\begin{split}
				a_{9}R^2\int_{\Omega}(\chi f)^{\frac{p}{2}+\alpha}|\nabla\eta|^2
				\leq&\
				a_{10}\alpha_0^2 \int_{\Omega}(\chi f)^{\frac{p}{2}+\alpha}\eta^{\frac{2\alpha+p}{\alpha_0+q-\frac{p}{2}+1}}\\
				\leq&\
				a_{10}\alpha_0^2 \left(\int_{\Omega}(\chi f)^{\alpha_0+q-\frac{p}{2}+1}\eta^2\right)^{\frac{\alpha_0+\frac{p}{2}}{\alpha_0+q-\frac{p}{2}+1}}V^{\frac{q-p+1}{\alpha_0+q-\frac{p}{2}+1}}\\
				\leq &\
				\frac{a_8\alpha_0R^2}{2}\left[\int_{\Omega}(\chi f)^{q-\frac{p}{2}+\alpha_0+1}\eta^2
				+
				\left(\frac{2a_{10}\alpha_0 }{a_8R^2}\right)^{\frac{\alpha_0+q-\frac{p}{2}+1}{q-p+1}}V\right].
			\end{split}
		\end{align}
		It follows from \eqref{equa24} and \eqref{315} that
		\begin{align}
			\label{equ:316}
			\begin{split}
				&\left(\int_{\Omega}(\chi f)^{\frac{n(\frac{p}{2}+\alpha_0)}{n-2}}\eta^{\frac{2n}{n-2}}\right)^{\frac{n-2}{n}}
				\\
				&\leq e^{ \alpha_0}V^{1-\frac{2}{n}}
				\left[\frac{a_8\alpha_0R^2}{2} \left(\frac{2a_{9}\alpha_0^2}{a_{8}R^2}\right)^{\frac{\alpha_0+q-\frac{p}{2}+1}{q-p+1}}
				+
				\frac{a_8\alpha_0R^2}{2}\left(\frac{2a_{10}\alpha_0 }{a_{9}R^2}\right)^{\frac{\alpha_0+q-\frac{p}{2}+1}{q-p+1}}
				\right].
			\end{split}
		\end{align}
		Since $\alpha_0\geq 1$, then it is not difficult to see from \eqref{equ:316} that
		\begin{align*}
			\left(\int_{\Omega}(\chi f)^{\frac{n}{n-2}(\frac{p}{2}+\alpha_0)}\eta^{\frac{2n}{n-2}}\right)^{\frac{n-2}{n}}
			\leq a_{11}^{\alpha_0}\alpha_0^3e^{ \alpha_0}V^{1-\frac{2}{n}}
			\left(\frac{\alpha_0^2}{ R^2}\right)^{\frac{\alpha_0+ \frac{p}{2} }{q-p+1}}.
		\end{align*}
		Recalling that $\alpha_1 = \frac{n}{n-2}\left(\alpha_0+\frac{p}{2}\right)$, we have
		\begin{align*}
			\|f\chi\|_{L^{\alpha_1}(B_{\frac{3R}{4}}(o))}
			&\leq a_{12}V^{\frac{1}{\alpha_1}}\left(\frac{\alpha^2_0}{R^2}\right)^{\frac{1}{q-p+1}}.
		\end{align*}
		Thus, we completed the proof of the lemma.
	\end{proof}

	\subsubsection{Moser Iteration (Proof of \thmref{t4})}
	Now we turn to use the Moser iteration to deduce the local $L^{\infty}$-estimate of $f$.
	\begin{proof}
		By discarding the second term on the left-hand side of \eqref{equa3.14}, we have		
		\begin{align}\label{equ31}
			\left(\int_{\Omega}(\chi f)^{\frac{n}{n-2}(\frac{p}{2}+\alpha)}\eta^{\frac{2n}{n-2}}\right)^{\frac{n-2}{n}}\leq\ &a_{9}e^{ \alpha_0} V^{-\frac{2}{n}}\left[R^{2}\int_M(\chi f)^{\frac{p}{2}+\alpha}|\nabla\eta|^2+K\alpha\int_{\Omega} (\chi f)^{\frac{p}{2}+\alpha}\eta^2 \right],
		\end{align}
		where $$K= K(\alpha) = (1+\sqrt{\kappa}R)^2 + R^2\|b\|^2_{1,\infty}\|u\|^{\frac{4rc}{2c-n}}_{k}(T\alpha)^{\frac{n }{2c-n}}.$$
		
		Define the sequence $\alpha_{l+1}=\alpha_l\frac{n}{n-2}$ for $l= 1,2,\cdots$. For the sequence of geodesic balls
		$$
		\Omega_l = B\left(o,\frac{R}{2}+\frac{R}{4^{l }}\right),\quad l=1, 2\ldots,
		$$
		we choose a sequence of cut-off functions $\eta_l\in C^{\infty}_0(\Omega_l)$ such that
		\begin{align*}
			\eta\equiv1 \text{ in }\Omega_{l+1}, \quad 0\leq\eta_l\leq 1\quad\text{and}\quad  |\nabla\eta_l|\leq\frac{C4^l}{R}.
		\end{align*}
		Taking $\alpha+\frac{p}{2}=\alpha_l$ and $\eta=\eta_l$ in \eqref{equ31}. Since
		\begin{align*}
			K(\alpha_l)\alpha_l \eta^2= & \left((1+\sqrt{\kappa}R)^2 +\|b\|^2_{1,\infty}\|u\|^{\frac{4rc}{2c-n}}_{k} R^2\left(T\left(\alpha_0+\frac{p}{2}\right)\left(\frac{n}{n-2}\right)^l\right)^{\frac{n }{2c-n}}\right)\alpha_l\\
			\leq &\left((1+\sqrt{\kappa}R)^2 +\|b\|^2_{1,\infty}\|u\|^{\frac{4rc}{2c-n}}_{k} R^2\left(T\left(\alpha_0+\frac{p}{2}\right) \right)^{\frac{n }{2c-n}}\right)\left(\frac{n}{n-2}\right)^{\frac{2cl}{2c-n}}\alpha_l\\
			\leq&\ a_{13}\alpha_0^2\left(\frac{n}{n-2}\right)^{\frac{2cl}{2c-n}}\alpha_l,
		\end{align*}
		we have
		\begin{align*}
			R^{2}|\nabla\eta_l|^2
			+ K(\alpha_l)\alpha_l \eta^2\leq &\ C^216^l +  a_{13}\alpha_0^2\left(\frac{n}{n-2}\right)^{\frac{2cl}{2c-n}}\alpha_l
			\\
			\leq &\  a_{14}^l\alpha_0^2\alpha_1.
		\end{align*}
		It follows from the above that
		\begin{align}\label{nashmoser}
			\begin{split}
				\left(\int_{\Omega_{l+1}}(\chi f)^{\alpha_{l+1}}\right)^{\frac{1}{\alpha_{l+1}}}
				\leq&\
				\left(a_{9}e^{ \alpha_0}V^{-\frac{2}{n}} \alpha_0^2\alpha_1\right)^{\frac{1}{\alpha_l}}
				a_{14}^{\frac{l}{\alpha_l}}\left(\int_{\Omega_l}(\chi f)^{\alpha_l}\right)^{\frac{1}{\alpha_l}}.
			\end{split}
		\end{align}
		By the facts
		\begin{align*}
			\sum_{l=1}^\infty\frac{1}{\alpha_l} = \frac{n}{2\alpha_1} \quad  \text{ and }\quad \sum_{l=1}^\infty\frac{l}{\alpha_l} = \frac{n^2}{4\alpha_1}< \frac{n(n-2)}{2p}
		\end{align*}
		and taking standard iteration, we obtain
		\begin{align}\label{iter4.1}
			\begin{split}
				\|\chi f\|_{L^{\infty}(B_{\frac{R}{2}}(o))}\leq&\
				\left(a_{9}e^{ \alpha_0}V^{-\frac{2}{n}} \alpha_0^2\alpha_1\right)^{\sum_{l=1}^\infty\frac{1}{\alpha_l}}
				a_{14}^{\sum_{l=1}^\infty\frac{l}{\alpha_l}}\|f\|_{L^{\alpha_1}(B_{\frac{3R}{4}}(o))}
				\\
				\leq&\ a_{15}
				V^{-\frac{1}{\alpha_1}}  \|f\|_{L^{\alpha_1}(B_{\frac{3R}{4}}(o))}.
			\end{split}
		\end{align}
		Furthermore, by \lemref{lpbound41} we can infer from \eqref{iter4.1} that
		\begin{align*}
			\|\chi f\|_{L^{\infty}(B_{\frac{R}{2}}(o))}
			\leq&\ a_{16} \left(\frac{\alpha^2_0}{R^2}\right)^{\frac{1}{q-p+1}} .
		\end{align*}
		In view of the upper bound of $\alpha_0$ given by \eqref{alpha0u}, we obtain
		\begin{align*}
			\sup_{B_{\frac{R}{2}}(o)}|\chi \nabla u|\leq a_{17}\left[\left(\frac{1+\sqrt{\kappa}R}{R}\right)^{\frac{1}{q-p+1}}
			+\left(R\|b\|^2_{1,\infty}\|u\|^{\frac{4rc}{2c-n}}_{k}T^{\frac{n }{2c-n}}\right)^{\frac{1}{q-p+1}}\right].
		\end{align*}
		Thus, we finish the proof of \thmref{t4}.
	\end{proof}
	
	\subsection{The case of $u\in L^\infty(B_R(o))$}
	\begin{prop}\label{lem5.1}
		Let $(M,g)$ be a complete Riemannian manifold with $\mathrm{Ric}_g\geq-(n-1)\kappa g$ for some constant $\kappa\geq0$. Suppose that $b\in W^{1,\infty}(B_R(o))$ is a real function and the constants $p$, $q$ and $r$ associated with \eqref{equ0} satisfy
		\begin{align}\label{condi44}
			p>1,\quad q>\max\{1,p-1\}\quad\text{and}\quad r\geq 1.
		\end{align}
		If $u\in C^1(B_R(o))\cap L^{\infty}(B_R(o))$ is a solution of \eqref{equ0}, then there exist two constants
		$$C=C(n,p,q,r)\quad\text{and}\quad  N =\left(\|b\|^2_{1,\infty}+\|b\|^2_{1,\infty}\|u\|^{2r}_{\infty}\right)^{\frac{1}{2(q-p+1)}}$$
		such that there holds true
		$$
		\sup_{B_{\frac{R}{2}}(o)}|\nabla u|\leq\max\left\{1,\, C \left[\left(\frac{1+\sqrt{\kappa}R}{R}\right)^{\frac{1}{q-p+1}}+N\right]\right\}
		$$
		where $V = {\rm{Vol} }(B_R(o))$ is the volume of the geodesic ball $B_R(o)$.
	\end{prop}
	In order to prove the proposition, we need to establish the following lemma.
	
	\subsubsection{$L^{\alpha_1}$ bound of gradient in a geodesic ball with radius $3R/4$}
	\begin{lem}
		Assume that $(M, g)$ satisfies the same conditions as the above proposition. Let $f$ and $\chi$ be the same as before. If $u\in C^1(B_R(o))\cap L^{\infty}(B_R(o))$ is a solution to equation \eqref{equ0} with $p,q$ and $r$ satisfying \eqref{condi44}, then there exists a positive constant $\alpha_1=\left(\alpha_0+\frac{p}{2}\right)\frac{n}{n-2}$ such that
		\begin{align*}
			\|\chi f\|_{L^{\alpha_1}(B_{\frac{3R}{4}}(o))}\leq a_{12}V^{\frac{1}{\alpha_1}}\left(\frac{\alpha_0^2}{ R^2} \right)^{\frac{1}{q-p+1}}.
		\end{align*}
		Here $$\alpha_0 = c_1\left(1+\sqrt{\kappa}R+R \|b\|_{1,\infty}\left(1+ \|u\|^{2r}_{\infty}\right)^{\frac{1}{2}}\right).$$
		where $c_1$ is a constant depending only on $n$, $p$ and $q$ (see \eqref{c1def}).
	\end{lem}
	
	\begin{proof}
		Since $u\in L^{\infty}(B_R(o))$, then we can infer from \eqref{equ32} that
		\begin{align}\label{441}
			\begin{split}
				&\frac{a_5}{\alpha}e^{-c_0(1+\sqrt{\kappa} R)}V^{\frac{2}{n}}R^{-2}\left(\int_{\Omega}(\chi f)^{\frac{n}{n-2} (\frac{p}{2}+\alpha)}\eta^{\frac{2n}{n-2}}\right)^{\frac{n-2}{n}}+\frac{1}{n-1}\int_{\Omega}(\chi f)^{\alpha+q-\frac{p}{2}+1}\eta^2\\
				\leq&\ a_7\left(\frac{(1+\sqrt{\kappa}R)^2}{R^{2}}+\|b\|^2_{1,\infty}+\|b\|^2_{1,\infty}\|u\|^{2r}_{\infty}\right)\int_{\Omega} (\chi f)^{\frac{p}{2}+\alpha}\eta^2 +\frac{a_3}{\alpha}\int_M(\chi f)^{\frac{p}{2}+\alpha}|\nabla\eta|^2,
			\end{split}
		\end{align}
		where $a_7$ depends on $n$, $M$, $p$, $q$ and $r$.
		
		Similarly, we choose
		\begin{align}\label{c1def}
			c_1 = \max\left\{c_0+1,\, \frac{4(n-1)a_2^2}{a_1 }\right\}
		\end{align}
		and set
		$$\alpha_0 = c_1\left(1+\sqrt{\kappa}R+R \|b\|_{1,\infty}\left(1+ \|u\|^{2r}_{\infty}\right)^{\frac{1}{2}}\right).$$
		The equation \eqref{441} becomes
		\begin{align}\label{equation:441}
			\begin{split}
				&\frac{a_5}{\alpha}e^{-\alpha_0}V^{\frac{2}{n}}R^{-2}\left(\int_{\Omega}(\chi f)^{\frac{n}{n-2} (\frac{p}{2}+\alpha)}\eta^{\frac{2n}{n-2}}\right)^{\frac{n-2}{n}}+\frac{1}{n-1}\int_{\Omega}(\chi f)^{\alpha+q-\frac{p}{2}+1}\eta^2\\
				\leq& a_7\frac{\alpha_0^2}{R^2}\int_{\Omega} (\chi f)^{\frac{p}{2}+\alpha}\eta^2 +\frac{a_3}{\alpha}\int_M(\chi f)^{\frac{p}{2}+\alpha}|\nabla\eta|^2.
			\end{split}
		\end{align}
		
		Now we denote $\alpha_1=\left(\alpha_0+\frac{p}{2}\right)\frac{n}{n-2}$ and next we will give a $L^{\alpha_1}$ bound of $f$. By choosing $\alpha=\alpha_0$ in \eqref{equation:441} and multiplying the both sides of \eqref{equation:441} by $\alpha_0/a_5$,  then we obtain	
		\begin{align}\label{442}
			\begin{split}
				&e^{-\alpha_0}V^{ \frac{2}{n}}\left(\int_{\Omega}(\chi f)^{\frac{n}{n-2}(\frac{p}{2}+\alpha_0)}\eta^{\frac{2n}{n-2}}\right)^{\frac{n-2}{n}}
				+a_{8}\alpha_0  R^{2}\int_{\Omega}(\chi f)^{\alpha_0+q-\frac{p}{2}+1}\eta^2
				\\
				\leq &\ a_{9}\left(
				\alpha^3_0 \int_{\Omega} (\chi f)^{\frac{p}{2}+\alpha_0}\eta^2
				+
				R^{2}\int_{\Omega}(\chi f)^{\frac{p}{2}+\alpha_0}|\nabla\eta|^2
				\right).
			\end{split}
		\end{align}
		Now, if
		$$
		f\geq \left( \frac{2a_9\alpha_0^2}{a_8R^2}\right)^{\frac{1}{q-p+1}},
		$$
		then we have
		$$
		a_{9}
		\alpha_0^3  (\chi f)^{\frac{p}{2}+\alpha_0}\eta^2
		\leq \frac{1}{2 } a_{8}\alpha_0  R^{2} (\chi f)^{\alpha_0+q-\frac{p}{2}+1}\eta^2.
		$$
		We decompose $\Omega$ into two subregions $\Omega_1$ and $\Omega_2$, where
		$$\Omega_1 =\left\{f\leq \left( \frac{2a_9\alpha_0^2}{a_8R^2}\right)^{\frac{1}{q-p+1}}\right\}$$
		and $\Omega_2$ is the complement of $\Omega_1$. We have
		\begin{align}\label{split400}
			\begin{split}
				a_{9}
				\alpha^3_0 \int_{\Omega} (\chi f)^{\frac{p}{2}+\alpha_0}\eta^2
				\leq&\ a_{9}
				\alpha^3_0 \int_{\Omega_1} (\chi f)^{\frac{p}{2}+\alpha_0}\eta^2+
				a_{9}\alpha^3_0 \int_{\Omega_2} (\chi f)^{\frac{p}{2}+\alpha_0}\eta^2
				\\
				\leq&\ a_{9}\alpha^3_0\left(\frac{2a_9\alpha_0^2}{a_8R^2} \right)^{\frac{\alpha_0+\frac{p}{2}}{q-p+1}}V + \frac{a_{8}\alpha_0R^{2}}{2}  \int_{\Omega}(\chi f)^{\alpha_0+q-\frac{p}{2}+1}\eta^2,
			\end{split}
		\end{align}
		where $V={\rm{Vol}}(\Omega)$ denote the volume of $\Omega$.
		
		Choose $\gamma\in C_0^{\infty}(B_R(o))$ such that
		$$
		\begin{cases}
			0\leq \gamma(x)\leq 1, \quad |\nabla\gamma(x)|\leq\frac{C }{R}, \quad &\forall x\in B_R(o);\\
			\gamma(x)\equiv 1,\quad &\forall x\in B_{\frac{3R}{4}}(o),
		\end{cases}
		$$
		and set $$\eta = \gamma^{\frac{\alpha_0+q-\frac{p}{2}+1}{q-p+1}}.$$
		Direct computation shows that
		\begin{align}\label{eta}
			\begin{split}
				|\nabla\eta| = \left|\left(\frac{\alpha_0+q-\frac{p}{2}+1}{q-p+1}\right)
				\eta^{\frac{\alpha_0+\frac{p}{2}}{\alpha_0+q-\frac{p}{2}+1}}\nabla\gamma\right|
				\leq \left(\frac{\alpha_0+q-\frac{p}{2}+1}{q-p+1}\right)\eta^{\frac{\alpha_0+\frac{p}{2}}{\alpha_0+q-\frac{p}{2}+1}}\frac{C }{R}.
			\end{split}
		\end{align}
		It follows from  \eqref{eta} and Young inequality that
		\begin{align}
			\begin{split}\label{inteta}
				a_{9}
				R^{2}\int_{\Omega}(\chi f)^{\frac{p}{2}+\alpha_0}|\nabla\eta|^2\leq		
				&\ a_{10}\alpha_0^2\int_{\Omega}(\chi f)^{\frac{p}{2}+\alpha_0}
				\eta^{\frac{2\alpha_0+p}{\alpha_0+q-\frac{p}{2}+1}}\\
				\leq &\
				a_{10}\alpha_0^2
				\left(\int_{\Omega}(\chi f)^{\alpha_0+q-\frac{p}{2}+1}\eta^2\right)^{\frac{\alpha_0+\frac{p}{2} }{\alpha_0+q-\frac{p}{2}+1}}
				V^{\frac{q-p+1 }{\alpha_0+q-\frac{p}{2}+1}}
				\\
				\leq &\
				\frac{a_{8}\alpha_0  R^{2}}{2 }\int_{\Omega}(\chi f)^{\alpha_0+q-\frac{p}{2}+1}\eta^2
				+
				\frac{a_{8}\alpha_0  R^{2}}{2 }
				\left(\frac{2a_{10}\alpha_0}{a_8R^2}\right)^{\frac{\alpha_0+q-\frac{p}{2}+1}{q-p+1}}V.
			\end{split}
		\end{align}
		Combining \eqref{split400} with \eqref{inteta}, we obtain
		\begin{align*}
			e^{- \alpha_0}V^{\frac{2}{n}}\left(\int_{\Omega}(\chi f)^{\frac{n}{n-2}(\frac{p}{2}+\alpha_0)}\eta^{\frac{2n}{n-2}}\right)^{\frac{n-2}{n}}
			\leq a_{10}\alpha_0^2
			\left(\frac{2a_{10}\alpha_0}{a_8R^2}\right)^{\frac{\alpha_0+ \frac{p}{2}}{q-p+1}}V + a_{9}\alpha^3_0\left(\frac{2a_9\alpha_0^2}{a_8R^2} \right)^{\frac{\alpha_0+\frac{p}{2}}{q-p+1}}V .
		\end{align*}
		Since $\alpha_0\geq 1$, the above inequality implies that
		\begin{align}\label{equation4.27}
			\left(\int_{\Omega}(\chi f)^{\frac{n}{n-2}(\frac{p}{2}+\alpha_0)}\eta^{\frac{2n}{n-2}}\right)^{\frac{n-2}{n}}
			\leq e^{\alpha_0}a_{11}^{\alpha_0}V^{1-\frac{2}{n}}\alpha^3_0\left(\frac{ K}{ R^2} \right)^{\frac{\alpha_0+\frac{p}{2}}{q-p+1}}.
		\end{align}	
		Taking $1/\left(\frac{p}{2}+\alpha_0\right)$ power on the both sides of \eqref{equation4.27}, we have
		\begin{align}\label{58}
			\|\chi f\|_{L^{\alpha_1}(B_{\frac{3R}{4}}(o))}\leq a_{12}V^{\frac{1}{\alpha_1}}\left(\frac{\alpha_0^2}{ R^2} \right)^{\frac{1}{q-p+1}}.
		\end{align}
		Here, $a_{12}$ depends only on $n$, $p$, $q$ and $r$. The proof of this lemma is completed.
	\end{proof}
	
	\subsubsection{Proof of \thmref{t5} or \propref{lem5.1} (Nash-Moser iteration)}
	\begin{proof}
		Deleting the second term on the left-hand side of inequality in \eqref{442}, we have
		\begin{align}\label{equation:4.29}
			\begin{split}
				&e^{-\alpha_0}V^{ \frac{2}{n}}\left(\int_{\Omega}(\chi f)^{\frac{n}{n-2}(\frac{p}{2}+\alpha )}\eta^{\frac{2n}{n-2}}\right)^{\frac{n-2}{n}}\\
				\leq &\ a_{9}\left(\alpha_0^2\alpha  \int_{\Omega} (\chi f)^{\frac{p}{2}+\alpha }\eta^2
				+ R^{2}\int_{\Omega}(\chi f)^{\frac{p}{2}+\alpha }|\nabla\eta|^2\right).
			\end{split}
		\end{align}
		For $l= 1,\, 2,\,\cdots$, we set $$\alpha_{l+1}=\alpha_l\frac{n}{n-2}=\alpha_1\left(\frac{n}{n-2}\right)^l\quad \mbox{and}\quad r_l=\frac{R}{2}+\frac{R}{4^{l }}.$$
		Let $\Omega_l=B_{r_l}(o)$. We choose $\eta_l\in C^{\infty}_0(\Omega_l)$ such that
		$$
		\begin{cases}
			0\leq \eta_l(x)\leq 1,& \forall x\in \Omega_l; \\
			\eta_l(x)\equiv1, &\forall x\in \Omega_{l+1} ;\\
			|\nabla\eta_l(x)|\leq \frac{C(n)4^l}{R},& \forall x\in \Omega_l.
		\end{cases}$$
		By substituting $\eta=\eta_l$ and $\alpha=\alpha_l-p/2$ into \eqref{equation:4.29}, we deduce that
		\begin{align}
			\label{444}
			\begin{split}
				e^{-\alpha_0}V^{ \frac{2}{n}}\left(\int_{\Omega_{l+1}}(\chi f)^{\alpha_{l+1}}\right)^{\frac{n-2}{n}}
				&\leq e^{-\alpha_0}V^{ \frac{2}{n}}\left(\int_{\Omega_{l}}(\chi f)^{\alpha_{l+1}}\eta_l^{\frac{2n}{n-2}}\right)^{\frac{n-2}{n}}\\
				&\leq  a_{9}\left(
				\alpha_0^2\alpha_l  \int_{\Omega_l} (\chi f)^{\alpha_l }\eta^2_l
				+
				R^{2}\int_{\Omega_l}(\chi f)^{\alpha_l }|\nabla\eta_l|^2
				\right).
			\end{split}
		\end{align}
		Noting the fact
		$$
		\alpha_0^2\alpha_l  \eta^2_l
		+
		R^{2} |\nabla\eta_l|^2\leq \alpha_0^2\alpha_1\left(\frac{n}{n-2}\right)^l+C^2(n)16^l\leq \alpha_1^3 a_{13}^l,
		$$
		we obtain
		\begin{align}\label{equation:4.32}
			\left(\int_{\Omega_{l+1}}(\chi f)^{\alpha_{l+1}}\right)^{\frac{n-2}{n}}
			\leq  e^{\alpha_0}V^{-\frac{2}{n}}a_{9 }\alpha_1^3 a_{13}^l
			\int_{\Omega_l} (\chi f)^{\alpha_l}.
		\end{align}
		Taking $1/\alpha_l$ power on both sides of \eqref{equation:4.32} yields
		\begin{align*}
			\|\chi f\|_{L^{\alpha_{l+1}}(\Omega_{l+1})}
			\leq
			\left(e^{\alpha_0}V^{-\frac{2}{n}}a_{9}\alpha_1^3  a_{13}^l \right)^{\frac{1}{\alpha_{l}}}\|\chi f\|_{L^{\alpha_{l}}(\Omega_{l})}.
		\end{align*}
		By an iteration to $\infty$ and the facts
		\begin{align*}
			\sum_{l=1}^\infty\frac{1}{\alpha_l} = \frac{n}{2\alpha_1} \quad  \text{ and }\quad
			\sum_{l=1}^\infty\frac{l}{\alpha_l} = \frac{n^2}{4\alpha_1}<  \frac{n(n-2)}{2p}
		\end{align*}
		we have
		\begin{align}\label{446}
			\begin{split}
				\|\chi f\|_{L^{\infty}(B_{\frac{R}{2}}(o))}
				\leq & \left(e^{\alpha_0}V^{-\frac{2}{n}}a_{9}\alpha^3_1 \right)^{\sum_{l=1}^{\infty}\frac{1}{\alpha_{l}}}a_{13}^{\sum_{l=1}^\infty \frac{l}{\alpha_l}}\|\chi f\|_{L^{\alpha_{1}}(\Omega_{1})}\\
				\leq &\ a_{14}V^{-\frac{1}{\alpha_1}}\|\chi f\|_{L^{\alpha_{1}}(\Omega_{1})}.
			\end{split}
		\end{align}
		Combining \eqref{446} with \eqref{58}, we obtain
		\begin{align*}
			\|\chi f\|_{L^{\infty}(B_{\frac{R}{2}}(o))}\leq
			a_{14}V^{-\frac{1}{\alpha_1}} a_{12}V^{\frac{1}{\alpha_1}} \left(\frac{\alpha_0^2}{ R^2} \right)^{\frac{1}{q-p+1}}
			\leq a_{15}\left(\frac{\alpha^2_0}{ R^2} \right)^{\frac{1}{q-p+1}}.
		\end{align*}
		It follows that
		\begin{align*}
			\sup_{B_{\frac{R}{2}}(o)}|\nabla u|
			\leq \max\left\{1,\, C\left[\left(\frac{1+\sqrt{\kappa}R}{R}\right)^{\frac{1}{q-p+1}}+\|b\|^{\frac{1}{q-p+1}}_{1,\infty}\left(1+ \|u\|^{2r}_{\infty}\right)^{\frac{1}{2(q-p+1)}}\right]\right\}.
		\end{align*}
		We has finished the proof of \thmref{t5}.
	\end{proof}
	
	\section{Properties of non-negative solutions to \eqref{equ0}}
	\subsection{$L^\infty$-estimate of non-negative solutions}
	In this section we consider the local $L^{\infty}$-estimate for any non-negative function $u$ satisfying the differential inequality
	\begin{align}\label{c0esti}
		\Delta_pu\geq -Cu^r,
	\end{align}
	where $0\leq r\leq p-1$ and $C$ is a positive constant. Since $b(x)\in L^\infty(B_R(o))$, any non-negative solution to \eqref{equ0} satisfies \eqref{c0esti}. We are ready to show \propref{c0estimate}. As before, we denote $\Omega=B_R(o)$.
	
	\begin{proof} Choose $\varphi = u^s\eta^t\chi$, where $\chi=\chi_{\{u\geq 1\}}$ is the characteristic function of the set $\{x:\, u(x)\geq 1\}$ and $\eta\in C^{\infty}_0(\Omega)$ is a non-negative cut-off function which will be determined later. Note that $\varphi\ge 0$. Multiplying $\varphi$ on both sides of \eqref{c0esti} and integrating over $\Omega$, we have
		\begin{align}
			\label{52}
			\begin{split}
				&\int_\Omega su^{s-1}\eta^t|\nabla u|^p\chi+\int_\Omega tu^s\eta^{t-1}|\nabla u|^{p-2}\la\nabla \eta,\nabla u\ra\chi\\
				\leq &C\int_\Omega u^{r+s}\eta^t\chi +\int_\Omega\mathrm{div}\left(|\nabla u|^{p-2}u^s\eta^t\nabla u\right)\chi.
			\end{split}
		\end{align}
		Since the outward pointing normal vector of the region $\{x:u(x)= 1\}$ is $-\nabla u/|\nabla u|$, we have
		\begin{align*}
			\int_\Omega\mathrm{div}(|\nabla u|^{p-2}u^s\eta^t\nabla u)\chi=\int_{\{u\geq 1\}}\mathrm{div}(|\nabla u|^{p-2}u^s\eta^t\nabla u)=-\int_{\{u=1\}}|\nabla u|^{p-1}u^s\eta^t\leq 0.
		\end{align*}
		Omitting the above boundary integral term and applying Cauchy inequality to \eqref{52} gives
		\begin{align}\label{301}
			\int_\Omega su^{s-1}\eta^t|\nabla u|^p\chi-\int_\Omega tu^s\eta^{t-1}|\nabla u|^{p-1}|\nabla \eta|\chi\leq C\int_\Omega u^{r+s}\eta^t\chi.
		\end{align}
		
		Now, using Young inequality to the second term on the left-hand side (denoted by $L_2^*$) of \eqref{301}, we can see that
		\begin{align}\label{left2}
			\begin{split}
				L_2^* = & -\int_\Omega\left(\epsilon|\nabla u|^{p-1}u^{\frac{(p-1)(s-1)}{p}}\eta^{\frac{(p-1)t}{p}}\right)
				\left(t\epsilon^{-1}u^{\frac{p+s-1}{p}}\eta^{\frac{t}{p}-1}|\nabla \eta|\chi\right)\\
				\geq&-\frac{p-1}{p}\epsilon^{\frac{p}{p-1}}\int_\Omega |\nabla u|^{p}u^{s-1}\eta^{t}\chi  -\frac{1}{p}t^p\epsilon^{-p}\int_\Omega u^{ p+s-1 }\eta^{t-p}|\nabla \eta|^p\chi.
			\end{split}
		\end{align}
		Choosing $\epsilon = s^{\frac{p-1}{p}}$ in \eqref{left2} and substituting \eqref{left2} into \eqref{301}, we obtain
		\begin{align}\label{302}
			\frac{s}{p}\int_\Omega u^{s-1}\eta^t|\nabla u|^p\chi-\frac{1}{p}t^ps^{1-p}\int_\Omega u^{ p+s-1 }\eta^{t-p}|\nabla \eta|^p\chi\leq\int_\Omega  u^{r+s}\eta^t\chi.
		\end{align}
		
		We will now address the first term on the left-hand side of \eqref{302}. Observing the following
		$$(a+b)^p\leq 2^p(a^p+b^p)$$ for any $a\geq 0$ and $b\geq 0$, we can simplify the expression
		\begin{align}\label{equa55}
			|\nabla(u^{\xi}\eta^{\theta})|^p\leq 2^p\left[\left(\xi u^{\xi-1}\eta^{\theta}|\nabla u|\right)^p + \left(\theta u^{\xi}\eta^{\theta-1}|\nabla\eta|\right)^p\right].
		\end{align}
		By picking the constants $\xi$ and $\theta$ in \eqref{equa55} such that
		$$p(\xi-1)=s-1 \quad \text{and} \quad \theta p=t,$$
		i.e.,
		$$
		\xi = \frac{p+s-1}{p}\quad\text{and}\quad\theta=\frac{t}{p},
		$$
		we can see that
		\begin{align}\label{equa56}
			\left(\frac{p}{2(p+s-1)}\right)^p\left|\nabla\left(u^{\frac{p+s-1}{p}}\eta^{\frac{t}{p}}\right)\right|^p\leq  u^{s-1}\eta^{t}|\nabla u|^p+\left(\frac{t}{p+s-1}\right)^p u^{p+s-1}\eta^{t-p}|\nabla\eta|^p.
		\end{align}
		Substituting \eqref{equa56} into \eqref{302}, we arrive at
		\begin{align}
			\label{57}
			\begin{split}
				&\frac{s}{p}\left(\frac{p}{2(p+s-1)}\right)^p\int_\Omega \left|\chi\nabla\left(u^{\frac{p+s-1}{p}}\eta^{\frac{t}{p}}\right)\right|^p
				- \frac{s}{p}\left(\frac{t}{p+s-1}\right)^p \int_\Omega u^{p+s-1}\eta^{t-p}|\nabla\eta|^p\chi\\
				\leq &\
				\frac{1}{p}t^ps^{1-p}\int_\Omega u^{ p+s-1 }\eta^{t-p}|\nabla \eta|^p\chi + C\int_\Omega u^{r+s}\eta^t\chi.
			\end{split}
		\end{align}
		Taking $t = p$ in \eqref{57}, we have
		\begin{align*}
			\frac{s}{p}\left(\frac{p}{2(p+s-1)}\right)^p\int_\Omega \left|\chi\nabla\left(u^{\frac{p+s-1}{p}}\eta^{\frac{t}{p}}\right)\right|^p
			\leq 2p^{p-1}s^{1-p}\int_\Omega u^{p+s-1} |\nabla \eta|^p\chi+\int_\Omega u^{r+s}\eta^p\chi,
		\end{align*}
		i.e.,
		\begin{align}\label{equa:5.8}
			\int_\Omega \left|\chi\nabla\left(u^{\frac{p+s-1}{p}}\eta^{\frac{t}{p}}\right)\right|^p
			\leq  C_1 \left[ \int_\Omega u^{ p+s-1 } |\nabla \eta|^p\chi+s^{p-1}\int_\Omega u^{r+s}\eta^p\chi\right],
		\end{align}
		where $C_1$ depends on $p$.
		
		By Sobolev inequality \eqref{lpsobolev}, we have
		\begin{align}
			\label{equa:5.9}
			\begin{split}
				\left\|u^{\frac{p+s-1}{p}}\eta \chi\right\|^p_{\frac{np}{n-p}}
				\leq&\
				C_2\left(\left\|\chi\nabla\left(u^{\frac{p+s-1}{p}}\eta \right)\right\|^p_p
				+ \left\| u^{\frac{p+s-1}{p}}\eta \chi \right\|^p_p\right)\\
				\leq &\
				C_1C_2 \left( \int_\Omega u^{ p+s-1 } |\nabla \eta|^p\chi+s^{p-1}\int_\Omega u^{r+s}\eta^p\chi\right) + C_2 \left\| u^{\frac{p+s-1}{p}}\eta \chi\right\|^p_p\\
				\leq&\
				C_3 \left[ \int_\Omega u^{ p+s-1 }\chi \left(|\nabla \eta|^p+1\right)+s^{p-1}\int_\Omega u^{r+s}\eta^p\chi\right],
			\end{split}
		\end{align}
		where the constants $C_2$ and $C_3$ depend on $M$, $p$ and $r$.
		
		Set the sequence $\rho_l = \frac{R}{2}+\frac{R}{4^l}$ for $l=1,\,2,\,\cdots$. For any ball $B_l=B(o,\rho_l)$, we choose a cut-off function $\eta_l\in C_0^{\infty}(B_{l})$ satisfying
		$$
		\begin{cases}
			\eta_l(x)\equiv 1, & x\in B_{l+1};\\
			0\leq \eta_l(x)\leq 1,|\nabla \eta_l|\leq \frac{C4^l}{R}, & x\in B_l.
		\end{cases}
		$$
		Since  $p-1\geq r$,  we have $$u^{r+s}\chi\leq u^{ s+p-1 }\chi.$$
		It follows from \eqref{equa:5.9} that
		\begin{align}\label{equa5.10}
			\|u \chi\|^{s+p-1}_{L^{\frac{n(s+p-1)}{n-p}}(B_{l+1})}\leq C_4\left(\frac{4^{lp}}{R^p}+s^{p-1}+1\right)\|u\chi\|_{L^{s+p-1}(B_l)}^{s+p-1}.
		\end{align}
		
		Now we define a sequence $s_l$ ($l=0,\,1,\, 2,\,\cdots$) by
		$$s_{l}=\left(\frac{n}{n-p}\right)^{l}p.$$
		We choose suitable $s$ such that $s+p-1=s_l$. It is easy to see that
		\begin{align}\begin{split}
				\label{512a}
				\|u\chi \|^{s_l}_{L^{s_{l+1}}(B_{l+1})}
				\leq\
				&C_5\left(\frac{4^{lp}}{R^p} +1+s_l^{p-1}\right)\|u\chi\|_{L^{s_l}(B_l)}^{s_l }\\
				\leq\
				&C_5\left[\frac{4^{lp}}{R^p} +1+p^{p-1}\left(\frac{n}{n-p}\right)^{l(p-1)}\right]\|u\chi\|_{L^{s_l}(B_l)}^{s_l}.
			\end{split}
		\end{align}
		We can choose some constant
		$$C_6>\max\left\{4^p,\, \left(\frac{n}{n-p}\right)^{p-1}\right\}$$
		such that
		\begin{align}\label{equa:5.11}
			\|u\chi \|_{L^{\alpha_{l+1}}(B_{l+1})}\leq  \left(\frac{1}{R^p} +p^{p-1}\right)^{\frac{1}{s_l}} C_6^{\frac{l}{s_l}}\|u\chi\|_{L^{s_l}(B_l)}.
		\end{align}
		
		Due to the facts
		$$
		\sum_{l=1}^\infty\frac{1}{s_l}=\frac{n-p}{p^2}\quad\text{and }\quad
		\sum_{l=1}^\infty\frac{l}{s_l}=\frac{n}{p^2},
		$$
		we can deduce from \eqref{equa:5.11} that		
		\begin{align*}
			\|u\chi\|_{L^{\infty}(B_{\frac{R}{2}}(o))}\leq & \left(\frac{1}{R^p} +p^{p-1}\right)^{\sum_{l=1}^\infty\frac{1}{s_l}} C_6^{\sum_{l=1}^\infty\frac{l}{s_l}}\|u\chi\|_{L^{s_0}(B_1)}
			\\
			&\leq C_7
			\left(\frac{1}{R^p} +p^{p-1}\right)^{\frac{n-p}{p^2}}  \|u\chi\|_{L^{p}(B_1)}
			\\
			&\leq
			C_8\left(\frac{1}{R}+1\right)^{\frac{n-p}{p}}\|u\chi\|_{L^{p}(B_1)}.
		\end{align*}
		So we obtain
		\begin{align}\label{5.14}
			\|u\|_{L^{\infty}(B_{\frac{R}{2}}(o))}\leq \max\left\{1,C_9\|u\|_{L^{p}(B_R(o))}\right\}.
		\end{align}
	\end{proof}
	It is worth to mention that $$C_9 = C_8\left(\frac{1}{R}+1\right)^{\frac{n-p}{p}}$$ depends on $\left(\frac{1}{R}+1\right)$, but $C_9$ is uniformly bounded as $R\to \infty$.
	
	\subsection{Proof of \thmref{thm17}}
	
	\begin{proof}
		\thmref{thm17} is a direct consequence of \propref{c0estimate} and \propref{lem5.1}. Indeed, by \propref{c0estimate} we know that any solution of equation \eqref{equ0} in a geodesic ball $B_R(o)$ satisfies
		\begin{align}\label{5.22}
			\sup_{B_{\frac{R}{2}}(o)} |u|\leq C,
		\end{align}
		where the constant $C=C(n,p,q,r, \|u\|_{L^p(B_R(o))})$.
		
		Now, since $u$ is also a solution of equation \eqref{equ0} in a geodesic ball $B_{\frac{R}{2}}(o)$ with radius $R/2$, by \propref{lem5.1} and \eqref{5.22} we have
		\begin{align}\label{equation:5.16}
			\sup_{B_{\frac{R}{4}}(o)}|\nabla u|\leq\max\left\{1,\, C\left[ \left(\frac{1+\sqrt{\kappa}R}{R}\right)^{\frac{1}{q-p+1}}+N \right] \right\},
		\end{align}
		where $N = N\left(n,p,q,r, \|b\|_{1,\infty}, \|u\|_{p}\right)$.
		Thus we complete the proof.
	\end{proof}
	
	If $u\in L^p(M, g)$ is a global non-positive solution of (\ref{equ0}) on a non-compact complete Riemannian manifold. By letting $R\to\infty$ in \eqref{equation:5.16}, we deduce that $|\nabla u|$ is uniformly bounded, which means $u$ has at most linear growth.
	\medskip\medskip
	
	\noindent {\it\bf{Acknowledgements}}: The authors are grateful to Jingchen Hu for the helpful discussion. The author Y. Wang is supported partially by NSFC (Grant No.11971400) and National Key Research and Development projects of China (Grant No. 2020YFA0712500).
	\medskip\medskip
	
	\bibliographystyle{unsrt}

\end{document}